\documentclass[11pt,reqno]{amsart}
\usepackage{amsaddr}
\usepackage{hyperref}
\usepackage{graphicx}
\usepackage{amssymb}
\usepackage{amsfonts}
\usepackage[]{amsmath}
\usepackage[]{epsfig}
\usepackage[]{pstricks}
\newpsobject{malla}{psgrid}{subgriddiv=1,griddots=10,gridlabels=6pt}
\usepackage[]{float}
\usepackage{setspace}
\newpsobject{malla}{psgrid}{subgriddiv=1,griddots=10,gridlabels=6pt}
\newtheorem{theorem}{Theorem}[section]
\newtheorem{lemma}{Lemma}[section]

\newtheorem{remark}{Remark}[section]
\newtheorem{proposition}{Proposition}[section]

\numberwithin{equation}{section}
\newcommand{\R}{\mathbb R}

\newcommand{\T}{{\mathbb T}}

\newcommand{\Z}{{\mathbb Z}}
\newcommand{\ft}{{\mathcal{F}}}

\newcommand{\Sch}{{\mathcal{S}}}
\newcommand{\supp}{{\mbox{supp}}}

\newcommand{\px}{\partial_x}
\newcommand{\pt}{\partial_t}

\def\norm#1{\|#1\|}
\def\bra#1{\langle#1\rangle}
\def\wt#1{\widetilde{#1}}
\def\wh#1{\widehat{#1}}
\def\set#1{\{#1\}}

\begin{document}

	\pagenumbering{arabic}	
\title[IBVP for the Kawahara equation on the half-line]{The initial-boundary value problem for the Kawahara equation on the half-line}

\author[M. Cavalcante]{M\'arcio Cavalcante}

\address{\emph{Instituto de Matem\'{a}tica, Universidade Federal de Alagoas\\ Macei\'o AL-Brazil}}
\email{marcio.melo@im.ufal.br}

\author[C. Kwak]{Chulkwang Kwak}

\address{\emph{Facultad de Matem\'{a}ticas, Pontificia Universidade Cat\'olica de Chile\\ Santiago-Chile}}
\email{chkwak@mat.uc.cl}
\thanks{C. Kwak is supported by FONDECYT de Postdoctorado 2017 Proyecto No. 3170067.}


\begin{abstract}
This paper concerns the initial-boundary value problem (IBVP) of the Kawahara equation posed on the right and left half-lines. We  prove the local well-posedness in the low regularity Sobolev space. We introduce the Duhamel boundary forcing operator, which is introduced by Colliander - Kenig \cite{CK} in the context of Airy group operators, to construct solutions on the whole line. We also give the bilinear estimate in $X^{s,b}$ space for $b < \frac12$, which is almost sharp compared to IVP of Kawahara equation \cite{CLMW2009, JH2009}.

\end{abstract}
\maketitle
\tableofcontents

\section{Introduction}
In this paper, we consider the following Kawahara equation\footnote{It is well-known of the form $\pt u +au\px u + b\px^3u - \px^5u=0$ for arbitrary constants $a,b \in \R$. We, however, use the form \eqref{eq:5kdv} for the simplicity, since the dominant dispersion term is the fifth-order term and constants do not affect our analysis.}:
\begin{equation}\label{eq:5kdv}
\pt u - \px^5 u + \px(u^2)=0.
\end{equation}
The Kawahara equation was first proposed by Kawahara \cite{Kawahara1972} describing solitary-wave propagation in media. Also, the Kawahara equation can be described in the theory of magneto-acoustic sound waves in plasma and in theory of shallow water waves with surface tension. For another physical background of Kawahara equation or in view of perturbed equation of KdV equation, see \cite{HS1988, PRG1988, Boyd1991} and references therein. 

Kawahara equation \eqref{eq:5kdv}, posed on the whole line $\R$, admits following three conservation laws:
\[M[u(t)] := \int u(t,x) \; dx = M[u(0)],\]
\begin{equation}\label{eq:L2}
E[u(t)] := \frac12 \int u^2(t,x) \; dx = E[u(0)]
\end{equation}
and
\begin{equation}\label{eq:hamiltonian}
H[u(t)] := \frac12 \int (\px^2 u)^2(t,x) \; dx - \frac13 \int u^3(t,x) \; dx = H[u(0)].
\end{equation}
Moreover, \eqref{eq:hamiltonian} allow us to represent \eqref{eq:5kdv} as the Hamiltonian equation:
\begin{equation*}
u_t = \partial_x \nabla_u H\left(u\left(t\right)\right),
\end{equation*}
where $\nabla_u$ is the variational derivative with respect to $u$, but not a completely integrable system in contrast to the KdV equation.

\subsection{Well-posedness results on $\R$.} 
The Cauchy problem for the Kawahara equation on $\R$ has been extensively studied, and here we only give some of previous works. The local well-posedness of Kawahara equation was first established by Cui and Tao \cite{CT2005}. They proved the Strichartz estimate for the fifth-order operator and obtained the local well-posedness in $H^s(\R)$ $s > 1/4$ as its application, which implies, in addition to the energy conservation law \eqref{eq:hamiltonian}, $H^2(\R)$ global well-posedness. Later,  Cui, Deng and Tao \cite{CDT2006} improved the previous result to the negative regularity Sobolev space $H^s(\R)$, $s > -1$, and Wang, Cui and Deng \cite{WCD2007} further improved to the lower regularity $s \ge -7/5$. In both paper, authors used Fourier restriction norm method, while more delicate analysis has been performed in the latter one than the former one. $L^2(\R)$ conservation law \eqref{eq:L2} allows the former result to extend the global one in $L^2(\R)$. In \cite{WCD2007}, they used \emph{I-method} to establish the global well-posedness in $H^s(\R)$, $s > -1/2$. In \cite{CLMW2009} and \cite{JH2009}, authors independently prove the local well-posedness in $H^s(\R)$, $s > -7/4$, while their methods are same, particularly, the Fourier restriction norm method in addition to Tao's $[K;Z]$-multiplier norm method. At the critical regularity Sobolev space $H^{-7/4}(\R)$, Chen and Guo \cite{CG2011} proved local and global well-posedness by using Besov-type critical space and \emph{I-method}. Kato \cite{Kato2011} proved the local well-posedness for $s \ge -2$ by modifying $X^{s,b}$ space and the ill-posedness for $s<-2$ in the sense that the flow map is discontinuous. We also refer to \cite{Huo2005, YL2010} and references therein for more results.

\subsection{The models on the half-lines and main results} We mainly consider the Kawahara equation on the right half-line $(0,\infty)$
\begin{equation}\label{kawahararight}
\begin{cases}
\pt u - \px^5 u + \px(u^2)=0, & (t,x)\in (0,T)  \times(0,\infty),\\
u(0,x)=u_0(x),                                   & x\in(0,\infty),\\
u(t,0)=f(t),\ u_x(t,0)=g(t)& t\in(0,T)
\end{cases}
\end{equation}
and on the left half-line $(-\infty, 0)$
\begin{equation}\label{kawaharaleft}
\begin{cases}
	\pt u - \px^5 u + \px(u^2)=0, & (t,x)\in  (0,T) \times (-\infty,0),\\
	u(0,x)=u_0(x),                                   & x\in(-\infty,0),\\
	u(t,0)=f(t),\ 	u_x(t,0)=g(t),\ u_{xx}(t,0)=h(t)& t\in(0,T).
\end{cases}
\end{equation}
The difference between the numbers of boundary conditions in \eqref{kawahararight} and \eqref{kawaharaleft} is motivated by integral identities on smooth solutions to the linear equation 
\begin{equation}\label{eq:linear 5kdv}
\partial_t u-\partial_x^5u=0.
\end{equation}
Indeed, for a smooth solution $u$ to \eqref{eq:linear 5kdv} and $T>0$, we have 
\begin{equation}\label{unico1}
\begin{split}
	\int_0^{\infty}u^2(T,x)dx=&\int_0^{\infty}u^2(0,x)dx -  \int_0^T(\partial_x^2u)^2(t,0) dt + 2\int_0^T\partial_x^3u(t,0)\partial_xu(t,0)dt\\
&	-2\int_0^T\partial_x^4u(t,0)u(t,0)dt
	\end{split}
\end{equation}
and
\begin{equation}\label{unico2}
\begin{split}
	\int_{-\infty}^{0}u^2(T,x)dx=&\int_{-\infty}^{0}u^2(0,x)dx + \int_0^T(\partial_x^2u)^2(t,0) dt - 2\int_0^T\partial_x^3u(t,0)\partial_xu(t,0)dt\\
&	+2\int_0^T\partial_x^4u(t,0)u(t,0)dt
	\end{split}
\end{equation}
by using \eqref{eq:linear 5kdv} and the integration by parts. Thus, under the condition 
\begin{equation}\label{eq:cond}
u(0,x)=0 \quad \mbox{for} \quad x>0, \qquad u(t,0)=\partial_xu(t,0)=0 \quad \mbox{for} \quad 0<t<T,
\end{equation} 
we can conclude from \eqref{unico1} that $u(T,x)=0$ for all $x>0$, while $u(t,x)\neq 0$ for $x<0$ exists under the same condition \eqref{eq:cond} (but $u(0,x)=0$ for $x < 0$) due to \eqref{unico2}. Indeed, even when $u_0(x) =0$ (zero initial data), $f = g = 0$ (first two zero boundary conditions) and $w=0$ (linear problem), there exist non-trivial $\gamma_j$, $j=1,2,3$, such that the solutions of the form \eqref{eq:solution2} with \eqref{eq:f2}--\eqref{eq:h2} in Section \ref{sec:main proof 2} satisfy the linear Kawahara equation \eqref{kawaharaleft} (see also \S 2.1 in \cite{Holmerkdv} for KdV equation case). Thus, from \eqref{unico2}, one can see that one more boundary condition for \eqref{kawaharaleft} is necessary in contrast to \eqref{kawahararight}.

It is well-known by Kenig, Ponce and Vega \cite{KPV1991} that the \emph{local smoothing effect} for the fifth-order linear group operator $e^{t\px^5}$
\[\|\partial_x^2e^{t\partial_x^5}\phi\|_{L_x^{\infty}L_t^2(\mathbb{R}_t)}\leq c\|\phi\|_{L_x^2(\mathbb{R})},\]
plays an important role in the low regularity local theory for the fifth-order equations. Moreover, it can be expressed as 
\begin{equation}\label{eq:lsm}
\|\partial_x^je^{t\partial_x^5}\phi\|_{L_x^{\infty}\dot{H}^{\frac{s+2-j}{5}}(\mathbb{R}_t)}\leq c\|\phi\|_{\dot{H}^{s}(\mathbb{R})},\ \text{for}\ j=0,1,2,
\end{equation}
which motivates the relation of regularities among initial and boundary data. Thus, we set the initial-boundary conditions for \eqref{kawahararight}
\begin{equation}\label{regularidade}
\begin{cases}
u_0\in H^s(\mathbb{R}^+),\; f(t)\in H^{\frac{s+2}{5}}(\mathbb{R}^+)\; \text{and}\;  g(t)\in H^{\frac{s+1}{5}}(\mathbb{R}^+),\\
u_0(0)=f(0),\; \text{when}\; \frac12 < s < \frac32,\\
u_0(0)=f(0),\ \partial_xu_0(0)=g(0),  \text{when}\; \frac32<s<\frac52
\end{cases}
\end{equation}
and for \eqref{kawaharaleft}
\begin{equation}\label{regularidadeleft}
 \begin{cases}
u_0\in H^s(\mathbb{R}^-), \; f(t)\in H^{\frac{s+2}{5}}(\mathbb{R}^+),\; g(t)\in H^{\frac{s+1}{5}}(\mathbb{R}^+)\; \text{and}\;  h(t)\in H^{\frac{s}{5}}(\mathbb{R}^+),\\
	u_0(0)=f(0),\; \text{when}\; \frac12 < s < \frac32,\\
u_0(0)=f(0),\ \partial_xu_0(0)=g(0),  \text{when}\; \frac32<s<\frac52.
 \end{cases}
 \end{equation}

When $s > \frac12$, the Sobolev embedding theorem allows that the trace map $u_0 \mapsto u_0(0)$ is well-defined on $H^s(\R^+)$. By the same reason, $f(0)$ is well-defined on $H^{\frac{s+2}{5}}(\R^+)$, since if $s > \frac12$, then $\frac{s+2}{5} > \frac12$. Since $u(0,0)$ means both quantities $u_0(0)$ and $f(0)$, they should be identical. Moreover, when $s > \frac32$, then both $s-1 > \frac12$ and $\frac{s+1}{5} > \frac12$ hold. Thus $\px u_0 (0)$ and $g(0)$ are well-defined on $H^{s-1}(\R^+)$ and $H^{\frac{s+1}{5}}(\R^+)$, respectively, and they must be identical. On the other hand, under the regularity condition $s < \frac52$, $\px^2 u \in H^{s-2}$ and $h \in H^{\frac{s}{5}}$ in \eqref{kawaharaleft} are not well-defined trace at zero. 

The main goal in the paper is to show the local well-posedness of \eqref{kawahararight} and \eqref{kawaharaleft} in the low regularity Sobolev space. In the previous works on the IBVP of the Kawahara equation and its related equations posed on the right half-line, e.g. \cite{San2003, Lar1, Fam2009, Lar2} and references therein, authors proved local and global well-posedness in the high regularity function spaces with exponential decay property, or weighted Sobolev space for at least nonnegative regularity. 

On the other hand, the IBVP for \eqref{kawaharaleft} is first considered in this paper as far as authors know it. The principal contribution in this paper is to study both IBVPs of \eqref{kawahararight} and \eqref{kawaharaleft} in the low regularity (including negative regularity) setting (compared to IVP of \eqref{eq:5kdv}) without any additional weight condition. 

We first state the main theorem for \eqref{kawahararight}.
\begin{theorem}\label{theorem1}
	Let $s \in (-\frac74,\frac52) \setminus \set{\frac12, \frac32}$. For given initial-boundary data $(u_0,f,g)$ satisfying \eqref{regularidade}, there exist a positive time $T$ depending on $\|u_0\|_{H^s(\mathbb{R}^+)}$, $\|f\|_{H^{\frac{s+2}{5}}(\mathbb{R}^+)}$ and $\|g\|_{H^{\frac{s+1}{5}}(\mathbb{R}^+)}$, and a unique solution $u(t,x) \in C((0 , T);H^s(\R^+))$ to \eqref{kawahararight}-\eqref{regularidade} satisfying
\[u \in C\bigl(\mathbb{R}^+;\; H^{\frac{s+2}{5}}(0,T)\bigr) \cap X^{s,b}((0,T) \times \R^+) \cap D^{\alpha}((0,T) \times \R^+) \; \mbox{ and } \; \partial_xu\in C\bigl(\R^+;\; H^{\frac{s+1}{5}}(0,T)\bigr)\]
for some $b(s) < \frac12$ and $\alpha(s) > \frac12$. Moreover, the map $(u_0,f,g)\longmapsto u$ is analytic from $H^s(\mathbb{R}^+)\times H^{\frac{s+2}{5}}(\mathbb{R}^+)\times H^{\frac{s+1}{5}}(\mathbb{R}^+)$ to  $C\big((0,T);\,H^s(\mathbb{R}^+)\big)$.
\end{theorem}
Moreover, we have the following theorem for \eqref{kawaharaleft}.
\begin{theorem}\label{theorem12}
	Let $s \in (-\frac74,\frac52) \setminus \set{\frac12, \frac32}$. For given initial-boundary data $(u_0,f,g,h)$ satisfying \eqref{regularidadeleft}, there exist a positive time $T$ depending on $\|u_0\|_{H^s(\mathbb{R}^-)}$, $\|f\|_{H^{\frac{s+2}{5}}(\mathbb{R}^+)}$, $\|g\|_{H^{\frac{s+1}{5}}(\mathbb{R}^+)}$ and $\|h\|_{H^{\frac{s}{5}}(\mathbb{R}^+)}$, and a unique solution $u(t,x) \in C((0 , T);H^s(\R^-))$ to \eqref{kawaharaleft}-\eqref{regularidadeleft} satisfying
\[\begin{aligned}
&u \in C\bigl(\mathbb{R}^-;\; H^{\frac{s+2}{5}}(0,T)\bigr) \cap X^{s,b}((0,T) \times \R^-) \cap D^{\alpha}((0,T) \times \R^-), \\
&\partial_xu\in C\bigl(\R^-;\; H^{\frac{s+1}{5}}(0,T)\bigr) \; \mbox{ and } \; \partial_x^2u\in C\bigl(\R^-;\; H^{\frac{s}{5}}(0,T)\bigr)
\end{aligned}\]
for some $b(s) < \frac12$ and $\alpha(s) > \frac12$. Moreover, the map $(u_0,f,g,h)\longmapsto u$ is analytic from $H^s(\mathbb{R}^-)\times H^{\frac{s+2}{5}}(\mathbb{R}^+)\times H^{\frac{s+1}{5}}(\mathbb{R}^+) \times H^{\frac{s}{5}}(\mathbb{R}^+)$ to  $C\big((0,T);\,H^s(\mathbb{R}^-)\big)$.
\end{theorem}

\begin{remark}
The lower bound of regularity range addressed in Theorems \ref{theorem1} and \ref{theorem12} follows from the following bilinear estimate 
\begin{equation}\label{bi.rem}
\norm{\px(uv)}_{X^{s,-b} \cap Y^{s,-b}} \le c\norm{u}_{X^{s,b} \cap D^{\alpha}}\norm{v}_{X^{s,b} \cap D^{\alpha}},
\end{equation}
where $Y^{s,b}$ is an intermediate norm in the iteration process, which will be introduced in Section \ref{sec:sol space}. It can be seen in \cite{Kato2011} that the bilinear estimate \eqref{bi.rem} fails when $s < -\frac74$ for $b > \frac12$ due to the \emph{high $\times$ high $\Rightarrow$ low} interaction component along the non-resonant phenomenon, which is the typical enemy of the bilinear estimates in the low regularity. Although the $X^{s,b}$ spaces for $b<\frac12$ is used to study the IBVP, the counter-example introduced in \cite{Kato2011} is still valid to our work, since the exponent $b$ is chosen very close to $\frac12$. Thus, one may say Theorems \ref{theorem1} and \ref{theorem12} are almost sharp compared with \cite{CLMW2009,JH2009}.

On the other hand, Chen and Guo \cite{CG2011} overcame the logarithmic divergence appearing in \eqref{bi.rem} (without $Y^{s,b}$ spaces) at the end-point regularity by using the Besov-type spaces ($X^{0,\frac12,1}$ spaces). Thereafter, Kato \cite{Kato2011} used an appropriate weight (weighted $X^{s,b}$ space introduced  by Bejenaru and Tao \cite{BT2006}, and further developed in \cite{BS2008, Kishimoto2009, Kishimoto2008, KT2010}) to extend the regularity range, where the bilinear estimate holds. However, the use of Besov-type or weighted $X^{s,b}$ spaces for the IBVP is not investigated as far as authors know, thus the local-well-posedness of the IBVP of \eqref{eq:5kdv} for $s \le -\frac74$ remains as an interesting open problem (also other equations, for example, KdV equation at $s =-\frac34$ compared to \cite{Guo2009, Kishimoto2009-1}).
\end{remark}

The proof is based on the Picard iteration method for a suitable extension of solutions. We first convert the IBVP of \eqref{eq:5kdv} posed in $\R^+ \times \R^+$ to the initial value problem (IVP) of \eqref{eq:5kdv} (integral equation formula) in the whole space $\R \times \R$ (see Section \ref{sec:Duhamel boundary forcing operator}) by using the Duhamel boundary forcing operator. The energy and nonlinear estimates (will be established in Sections \ref{sec:energy} and \ref{sec:bilinear}, respectively) allow us to apply the Picard iteration method to IVP of \eqref{eq:5kdv}, and hence we can complete the proof. The new ingredients here are the Duhamel boundary forcing operator for the fifth-order linear equation and its analysis, and the bilinear estimate.

It is well-known that the Kawahara equation \eqref{eq:5kdv} enjoys the scaling symmetry: if $u$ is a solution to \eqref{eq:5kdv}, $u_{\lambda}$ defined by
\[u_{\lambda}(t,x) := \lambda^4u(\lambda^5t,\lambda x), \qquad \lambda > 0\]
is a solution to \eqref{eq:5kdv} as well. A straightforward calculation gives
\begin{equation}\label{eq:scaling}
\begin{aligned}
\norm{u_{0,\lambda}}_{H^s} &+ \norm{f_{\lambda}}_{H^{\frac{s+2}{5}}} + \norm{g_{\lambda}}_{H^{\frac{s+1}{5}}} + \norm{h_{\lambda}}_{H^{\frac{s}{5}}} \\
=&\lambda^{\frac72}\bra{\lambda}^s\norm{u_0}_{H^s} + \lambda^{\frac32}\bra{\lambda}^{s+2}\norm{f}_{H^{\frac{s+2}{5}}} + \lambda^{\frac52}\bra{\lambda}^{s+1}\norm{g}_{H^{\frac{s+1}{5}}} + \lambda^{\frac72}\bra{\lambda}^{s}\norm{f}_{H^{\frac{s}{5}}},
\end{aligned}
\end{equation}
which allows us to regard IBVPs \eqref{kawahararight} and \eqref{kawaharaleft} as small data problems.

\begin{remark}
The analysis of Duhamel boundary forcing operator established in this paper can be applied to both the fifth-order equation with different nonlinearities and the same equation on different boundary regions, for instance, a segment in the right / left half-line, a metric star graph, and so on. See, for instance, our upcoming work \cite{CK2018-2} for the IBVP of the fifth-order KdV-type equations, whose nonlinearities do not allow the scaling argument in contrast to this paper. 
\end{remark}

\subsection{Organization of the paper} 
The rest of paper is organized as follows: In Section \ref{sec:pre}, we introduce some function spaces defined on the half line and construct the solution spaces. In Section \ref{sec:Duhamel boundary forcing operator}, we introduce the boundary forcing operator for the fifth-order KdV-type equation. In Sections \ref{sec:energy} and \ref{sec:bilinear}, we show the energy estimates and the bilinear estimates, respectively. In Sections \ref{sec:main proof 1} and \ref{sec:main proof 2}, we complete the proofs of Theorems \ref{theorem1} and \ref{theorem12}, respectively.

\section{Preliminaries}\label{sec:pre}
Let $\R^+ = (0,\infty)$. For positive real numbers $x,y \in \R^+$, we mean $x \lesssim y$ by $x \le Cy$ for some $C>0$. Also, $x \sim y$ means $x \lesssim y$ and $y\lesssim x$. Similarly, $\lesssim_s$ and $\sim_s$ can be defined, where the implicit constants depend on $s$. Let $a_1,a_2,a_3 \in \R$. The quantities $a_{max} \ge a_{med} \ge a_{min}$ can be conveniently defined to be the maximum, median and minimum values of $a_1,a_2,a_3$ respectively.

Throughout the paper, we fix a cut-off function (even function)
\begin{equation}\label{eq:cutoff}
\psi \in C_0^{\infty}(\mathbb{R}) \quad \mbox{such that} \quad 0 \le \psi \le1, \quad  \psi \equiv 1 \; \mbox{ on } \; [-1,1], \quad \psi \equiv 0, \; |t| \ge 2. 
\end{equation}

\subsection{Sobolev spaces on the half line}
For $s\geq 0$, we say $f \in H^s(\mathbb{R}^+)$ if there exists $F \in H^s(\R)$ such that $f(x)=F(x)$ for $x>0$, in this case we set $\|f\|_{H^s(\mathbb{R}^+)}=\inf_{F}\|F\|_{H^{s}(\mathbb{R})}$. For $s \in \R$, we say $f \in H_0^s(\mathbb{R}^+)$ if there exists $F \in H^s(\R)$ such that $F$ is the extension of $f$ on $\R$ and $F(x) = 0$ for $x<0$. In this case, we set $\norm{f}_{H^s_0(\R^+)} = \norm{F}_{H^s(\R)}$.
For $s<0$, we define $H^s(\mathbb{R}^+)$ as the dual space of $H_0^{-s}(\mathbb{R}^+)$.

We also set $C_0^{\infty}(\mathbb{R}^+)=\{f\in C^{\infty}(\mathbb{R});\, \supp f \subset [0,\infty)\}$, and define $C_{0,c}^{\infty}(\mathbb{R}^+)$ as the subset of $C_0^{\infty}(\mathbb{R}^+)$, whose members have a compact support on $(0,\infty)$. We remark that $C_{0,c}^{\infty}(\mathbb{R}^+)$ is dense in $H_0^s(\mathbb{R}^+)$ for all $s\in \mathbb{R}$.

We finish this subsection with stating elementary properties of the Sobolev spaces.

\begin{lemma}\label{sobolevh0}
	For $-\frac{1}{2}<s<\frac{1}{2}$ and $f\in H^s(\mathbb{R})$, we have
	\begin{equation}\label{eq:0}
	\|\chi_{(0,\infty)}f\|_{H^s(\mathbb{R})}\leq c \|f\|_{H^s(\mathbb{R})}.
	\end{equation}
\end{lemma}

\begin{proof}
The proof of \eqref{eq:0} for $0 \le s < 1/2$ immediately follows from Proposition 3.5 in \cite{JK1995}. For $-1/2 < s <0$, for $f \in H^s(\R)$, the duality of $H^s(\R)$ and \eqref{eq:0} for $0 < -s < \frac12$ yields $\chi_{(0,\infty)}f \in H^s(\R)$. Hence, from the definitions of $H^s(\R^+)$ and $H_0^s(\R^+)$, and Proposition 2.7 in \cite{CK}, we have
\[
\norm{\chi_{(0,\infty)}f}_{H^s(\R)} = \norm{\chi_{(0,\infty)}f}_{H^{s}_0(\R^+)} \sim \norm{\chi_{(0,\infty)}f}_{H^{s}(\R^+)} \lesssim \norm{f}_{H^s(\R)}.
\]
\end{proof}

\begin{lemma}\label{sobolev0}
	If $0\leq s<\frac{1}{2}$, then $\|\psi f\|_{H^s(\mathbb{R})}\leq c \|f\|_{\dot{H}^{s}(\mathbb{R})}$ and $\|\psi f\|_{\dot{H}^{-s}(\mathbb{R})}\leq c \|f\|_{H^{-s}(\mathbb{R})}$ , where the constant $c$ depends only on $s$ and $\psi$. 
\end{lemma}
Remark that Lemma \ref{sobolev0} is equivalent that $\|f\|_{H^s(\mathbb{R})} \sim \|f\|_{\dot{H}^{s}(\mathbb{R})}$ for $-\frac12 < s < \frac12$ where $f \in H^s$ with $\supp f \subset [0,1]$.
\begin{proof}
The proof follows from the Cauchy-Schwarz inequality and the smoothness of $\psi$. We remark that the regularity range follows from the fact
\[\int_0^1 \frac{x^{\alpha}}{dx} < \infty \quad \mbox{for} \quad -1 < \alpha.\]
\end{proof}

\begin{lemma}[Proposition 2.4 in \cite{CK}]\label{alta}
	If $\frac{1}{2}<s<\frac{3}{2}$ the following statements are valid:
	\begin{enumerate}
\item [(a)] $H_0^s(\R^+)=\big\{f\in H^s(\R^+);f(0)=0\big\},$\medskip
\item [(b)] If  $f\in H^s(\R^+)$ with $f(0)=0$, then $\|\chi_{(0,\infty)}f\|_{H_0^s(\R^+)}\leq c \|f\|_{H^s(\R^+)}$.
	\end{enumerate}
\end{lemma}

\begin{lemma}[Proposition 2.5. in \cite{CK}]\label{cut}
	Let $-\infty<s<\infty$ and $f\in  H_0^s(\mathbb{R}^+)$. For the cut-off function $\psi$ defined in \eqref{eq:cutoff}, we have $ \|\psi f\|_{H_0^s(\mathbb{R}^+)}\leq c \|f\|_{H_0^s(\mathbb{R}^+)}$.
\end{lemma}

\subsection{Solution spaces}\label{sec:sol space}
For $f \in \Sch '(\R^2) $ we denote by $\wt{f}$ or $\ft (f)$ the Fourier transform of $f$ with respect to both spatial and time variables
\[\wt{f}(\tau , \xi)=\int _{\R ^2} e^{-ix\xi}e^{-it\tau}f(t,x) \;dxdt .\]
Moreover, we use $\ft_x$ and $\ft_t$ to denote the Fourier transform with respect to space and time variable respectively (also use $\wh{\;}$ for both cases). 

For $s,b\in \mathbb{R}$, we introduce the classical Bourgain spaces $X^{s,b}$ associated to \eqref{eq:5kdv} as the completion of $\Sch'(\mathbb{R}^2)$ under the norm
\[\norm{f}_{X^{s,b}}^2 = \int_{\R^2} \bra{\xi}^{2s}\bra{\tau - \xi^5}^{2b}|\wt{f}(\tau,\xi)|^2 \; d\xi d\tau, \]
where $\bra{\cdot} = (1+|\cdot|^2)^{1/2}$. The Fourier restriction norm method was first implemented in its current form by Bourgain \cite{Bourgain1993} and further developed by Kenig, Ponce and Vega \cite{KPV1996} and Tao \cite{Tao2001}. The following is one of the basic properties of $X^{s,b}$ estimates:
\begin{lemma}[Lemma 2.11 in \cite{Tao2006}]\label{lem:Xsb}
Let $\psi(t)$ be a Schwartz function in time. Then, we have 
\[\norm{\psi(t)f}_{X^{s,b}} \lesssim_{\psi,b} \norm{f}_{X^{s,b}}.\]
\end{lemma}

As well-known, the $X^{s,b}$ space with $b> \frac12$ is well-adapted to study the IVP of dispersive equations. However, in the study of the IBVP, the standard argument cannot be applied directly due to the following two reasons: First, the control of (derivatives) time trace norms of the Duhamel parts requires us to introduce modified $X^{s,b}$-type spaces, since the full regularity range cannot be covered (see Lemma \ref{duhamel} (b)). In order to overcome this weakness, we define the (time-adapted) Bourgain space $Y^{s,b}$ associated to \eqref{eq:5kdv} as the completion of $S'(\mathbb{R}^2)$ under the norm
\[\norm{f}_{Y^{s,b}}^2 = \int_{\R^2} \bra{\tau}^{\frac{2s}{5}}\bra{\tau - \xi^5}^{2b}|\wt{f}(\tau,\xi)|^2 \; d\xi d\tau.\]
Second, the Duhamel boundary forcing operator in the study of the IBVP requires us to take the exponent $b$ of the standard $X^{s,b}$ space in the range $(0,\frac12)$. It forces us to use the low frequency localized $X^{0,b}$-type space with $b > \frac12$ in the nonlinear estimates. Hence, we define $D^{\alpha}$ space as the completion of $\Sch'(\mathbb{R}^2)$ under the norm
\[\norm{f}_{D^{\alpha}}^2 = \int_{\R^2} \bra{\tau}^{2\alpha}\mathbf{1}_{\set{\xi : |\xi| \le 1}}(\xi)|\wt{f}(\tau,\xi)|^2 \; d\xi d\tau,\]
where $\mathbf{1}_A$ is the characteristic functions on a set $A$. 

Let $\Z_+ = \Z \cap [0,\infty)$. For $k \in \Z_+$, we set
\[I_0 = \set{\xi \in \R : |\xi| \le 2} \hspace{1em} \mbox{ and } \hspace{1em} I_k = \set{\xi \in \R : 2^{k-1} \le |\xi| \le 2^{k+1}}, \hspace{1em} k \ge 1.\]
Let $\eta_0: \R \to [0,1]$ denote a smooth bump function supported in $ [-2,2]$ and equal to $1$ in $[-1,1]$. For $k \in \Z_+ $, let 
\[\chi_0(\xi) = \eta_0(\xi), \hspace{1em} \mbox{and} \hspace{1em} \chi_k(\xi) = \eta_0(\xi/2^k) - \eta_0(\xi/2^{k-1}), \hspace{1em} k \ge 1,\]
which is supported in $I_k$, and
\[\chi_{[k_1,k_2]}=\sum_{k=k_1}^{k_2} \chi_k \quad \mbox{ for any} \ k_1 \le k_2 \in \Z_+ .\]
$\{ \chi_k \}_{k \in \Z_+}$ is the inhomogeneous decomposition function sequence to the frequency space. For $k\in \Z_+$, let $P_k$ denote the
operators on $L^2(\R)$ defined by $\widehat{P_kv}(\xi)=\chi_k(\xi)\wh{v}(\xi)$. For $l\in \Z_+$, let
\[P_{\le l}=\sum_{k \le l}P_k, \quad P_{\ge l}=\sum_{k \ge l}P_k.\]
For the modulation decomposition, we use the multiplier $\eta_j$, but the same as $\eta_j(\tau-\xi^5) = \chi_j(\tau-\xi^5)$. For $k,j \in \Z_+$, let
\[D_{k,j}=\{(\tau,\xi) \in \R^2 : \tau - \xi^5 \in I_j, \xi \in I_k \}, \hspace{2em} D_{k,\le j}=\cup_{l\le j}D_{k,l}.\]

Note that the Littlewood-Paley theory allows that
\begin{equation}\label{eq:dyadic X}
\norm{f}_{X^{s,b}}^2 \sim \sum_{k\ge0}\sum_{j\ge0}2^{2sk}2^{2bj}\norm{\eta_j(\tau-\xi^5)\chi_k(\xi)\wt{f}(\tau,\xi)}_{L^2}^2
\end{equation}
and
\[\norm{f}_{D^{\alpha}}^2 \sim \norm{P_0f}_{X^{0,\alpha}}^2.\]

Now we set the solution space denoted by $Z_{\ell}^{s,b,\alpha}$ with the following norm\footnote{The $Y^{s,b}$ space is nothing but the intermediate norm in the Picard iteration argument (see Lemma \ref{duhamel} (b) and Section \ref{sec:bilinear}).}:
\[\norm{f}_{Z_{\ell}^{s,b,\alpha}(\R^2) }= \sup_{t \in \R} \norm{f(t,\cdot)}_{H^s} + \sum_{j=0}^{\ell}\sup_{x \in \R} \norm{\px^jf(\cdot,x)}_{H^{\frac{s+2-j}{5}}} + \norm{f}_{X^{s,b} \cap D^{\alpha}},\]
for some $\ell \in \Z_+$. Remark that $Z_{1}^{s,b,\alpha}$ and  $Z_{2}^{s,b,\alpha}$ will be used for the right-half line and the left-half line problems, respectively (see Sections \ref{sec:main proof 1} and \ref{sec:main proof 2}). The spatial and time restricted space of $Z_{\ell}^{s,b,\alpha}(\R^2)$ is defined by the standard way:
\[Z_{\ell}^{s,b,\alpha}((0 , T)\times \R^+) = Z_{\ell}^{s,b,\alpha} \Big|_{(0 , T)\times \R^+}\]
equipped with the norm
\[\norm{f}_{Z_{\ell}^{s,b,\alpha}((0 , T)\times \R^+) } = \inf\limits_{g \in Z_{\ell}^{s,b,\alpha}} \set{\norm{g}_{Z_{\ell}^{s,b,\alpha}} : g(t,x) = f(t,x) \; \mbox{ on } \; (0,T) \times \R^+}. \]

\subsection{Riemann-Liouville fractional integral}
In this section, we just give a summary of the Riemann-Liouville fractional integral operator, see \cite{CK, Holmerkdv} for more details. Let $t_+$ be a function defined by
\[t_+ = t \quad \mbox{if} \quad t > 0, \qquad t_+ = 0  \quad \mbox{if} \quad t \le 0.\]
We also define $t_- = (-t)_+$. The tempered distribution $\frac{t_+^{\alpha-1}}{\Gamma(\alpha)}$ is defined as a locally integrable function for Re $\alpha>0$ by
\begin{equation*}
	\left \langle \frac{t_+^{\alpha-1}}{\Gamma(\alpha)},\ f \right \rangle=\frac{1}{\Gamma(\alpha)}\int_0^{\infty} t^{\alpha-1}f(t)dt.
\end{equation*}
It follows that
\begin{equation*}
	\frac{t_+^{\alpha-1}}{\Gamma(\alpha)}=\partial_t^k\left( \frac{t_+^{\alpha+k-1}}{\Gamma(\alpha+k)}\right),
\end{equation*}
for all $k\in\mathbb{N}$. This expression can be used to extend the definition of $\frac{t_+^{\alpha-1}}{\Gamma(\alpha)}$ to all $\alpha \in \mathbb{C}$ in the sense of distributions. A change of contour  calculation shows the Fourier transform of $\frac{t_+^{\alpha-1}}{\Gamma(\alpha)}$ as follows:
\begin{equation}\label{transformada}
	\left(\frac{t_+^{\alpha-1}}{\Gamma(\alpha)}\right)^{\widehat{}}(\tau)=e^{-\frac{1}{2}\pi i \alpha}(\tau-i0)^{-\alpha},
\end{equation}
where $(\tau-i0)^{-\alpha}$ is the distributional limit. One can also rewrite \eqref{transformada} as follows for $\alpha \notin \Z$: 
\begin{equation}\label{transformada1}
	\left(\frac{t_+^{\alpha-1}}{\Gamma(\alpha)}\right)^{\widehat{}}(\tau)=e^{-\frac12 \alpha \pi i}|\tau|^{-\alpha}\chi_{(0,\infty)}+e^{\frac12 \alpha \pi i}|\tau|^{-\alpha}\chi_{(-\infty,0)}.
\end{equation}
Note from \eqref{transformada} and \eqref{transformada1} that
\begin{equation}\label{transformada2}
(\tau-i0)^{-\alpha} = |\tau|^{-\alpha}\chi_{(0,\infty)}+e^{\alpha \pi i}|\tau|^{-\alpha}\chi_{(-\infty,0)}.
\end{equation}
For $f\in C_0^{\infty}(\mathbb{R}^+)$, we define
\begin{equation*}
	\mathcal{I}_{\alpha}f=\frac{t_+^{\alpha-1}}{\Gamma(\alpha)}*f.
\end{equation*}
Thus, for $\mbox{Re }\alpha>0$, we have
\begin{equation}\label{eq:IO}
	\mathcal{I}_{\alpha}f(t)=\frac{1}{\Gamma(\alpha)}\int_0^t(t-s)^{\alpha-1}f(s) \; ds,
\end{equation}
and basic properties $\mathcal{I}_0f=f$, $\mathcal{I}_1f(t)=\int_0^tf(s) \; ds$, $\mathcal{I}_{-1}f=f'$ and $\mathcal{I}_{\alpha}\mathcal{I}_{\beta}=\mathcal{I}_{\alpha+\beta}$.
\begin{lemma}[Lemma 2.1 in \cite{Holmerkdv}]
	If $f\in C_0^{\infty}(\mathbb{R}^+)$, then $\mathcal{I}_{\alpha}f\in C_0^{\infty}(\mathbb{R}^+)$, for all $\alpha \in \mathbb{C}$.
\end{lemma}
\begin{lemma}[Lemma 5.3 in \cite{Holmerkdv}]\label{lio}
	If $0\leq \mathrm{Re} \ \alpha <\infty$ and $s\in \mathbb{R}$, then $\|\mathcal{I}_{-\alpha}h\|_{H_0^s(\mathbb{R}^+)}\leq c \|h\|_{H_0^{s+\alpha}(\mathbb{R}^+)}$, where $c=c(\alpha)$.
\end{lemma}
\begin{lemma}[Lemma 5.4 in \cite{Holmerkdv}]
	If $0\leq \mathrm{Re}\ \alpha <\infty$, $s\in \mathbb{R}$ and $\mu\in C_0^{\infty}(\mathbb{R})$, then
	$\|\mu\mathcal{I}_{\alpha}h\|_{H_0^s(\mathbb{R}^+)}\leq c \|h\|_{H_0^{s-\alpha}(\mathbb{R}^+)},$ where $c=c(\mu, \alpha)$.
\end{lemma}

\subsection{Oscillatory integral}
Let 
\begin{equation}\label{eq:oscil}
B(x)=\frac{1}{2\pi}\int_{\R}e^{ix\xi}e^{i \xi^5} \; d\xi.
\end{equation} 

We first calculate $B(0)$ and $B^{(n)}(0)$, $n=1,2,3$. A change of variable ($\eta = \xi^5$), we have 
\[
B(0)=\frac{1}{2\pi}\int_{\R}e^{i \xi^5} \;d\xi=\frac{1}{10\pi}\int_{\R}e^{i \eta}\eta^{-\frac45}\;d\eta.
\]
We separate the interval into $(0,\infty)$ and $(-\infty,0)$, and apply the change of contour to each interval to obtain
\[B(0) = \frac{1}{10\pi}\left(e^{\frac{i\pi}{10}}\Gamma(1/5) + e^{-\frac{i\pi}{10}}\Gamma(1/5) \right).\]
Hence, by using the property of the gamma function $\Gamma(z)\Gamma(1-z) = \frac{\pi}{\sin (z\pi )}$, we obtain
\begin{equation}\label{B(0)}
B(0) = \frac{\cos\left(\frac{\pi}{10}\right)}{5\pi}\Gamma(1/5) = \frac{\cos\left(\frac{\pi}{10}\right)}{5\sin\left(\frac{\pi}{5}\right)\Gamma(4/5)}.
\end{equation}
A direct calculation gives
\begin{equation}\label{eq:oscil0}
B^{(n)}(x)=\frac{1}{2\pi}\int_{\R}(i\xi)^ne^{ix\xi}e^{i \xi^5} \; d\xi = \frac{1}{10\pi}\int_{\R}e^{i \eta}\eta^{\frac{n-4}{5}}\;d\eta., \qquad n = 1, 2, 3, \cdots.
\end{equation} 
Similarly, we have
\begin{equation}\label{B'(0)}
B'(0) = -\frac{\cos\left(\frac{3\pi}{10}\right)}{5\pi}\Gamma(2/5) = -\frac{\cos\left(\frac{3\pi}{10}\right)}{5\sin\left(\frac{2\pi}{5}\right)\Gamma(3/5)},\
\end{equation}
\begin{equation}\label{B''(0)}
B''(0) = -\frac{\cos\left(\frac{3\pi}{10}\right)}{5\pi}\Gamma(3/5) = -\frac{\cos\left(\frac{3\pi}{10}\right)}{5\sin\left(\frac{2\pi}{5}\right)\Gamma(2/5)}
\end{equation}
and 
\begin{equation}\label{B'''(0)}
B^{(3)}(0) = \frac{\cos\left(\frac{\pi}{10}\right)}{5\pi}\Gamma(4/5) = \frac{\cos\left(\frac{\pi}{10}\right)}{5\sin\left(\frac{\pi}{5}\right)\Gamma(1/5)},
\end{equation}
thanks to $\sin(z\pi) = \sin((1-z)\pi)$. 

Moreover, from \eqref{eq:oscil}, Fubini theorem and the definition of the Heaviside function $H(x)$ (see \cite{Fr} for the details), we have
\[\int_0^{\infty} B(y) \; dy = \frac{1}{2\pi}\int \int_0^{\infty} e^{iy\xi} \; dy e^{i\xi^5} \; d\xi = \frac{1}{2\pi}\int \wh{H}(-\xi)e^{i\xi^5} \; d\xi.\]
We use the change of variable ($\xi \mapsto -\xi$) the properties of the Heaviside function ($\wh{H}(\xi) = \mbox{p.v.}\frac{1}{i\xi} + \pi \delta(\xi)$) to obtain
\[\frac{1}{2\pi}\int \wh{H}(-\xi)e^{i\xi^5} \; d\xi = \frac{1}{2\pi}\int (\mbox{p.v.}\frac{1}{i\xi} + \pi \delta(\xi))e^{-i\xi^5} \; d\xi.\]

The change of variable ($\xi \mapsto \eta^{\frac15}$) and the fact $\ft_x[\mbox{p.v.}\frac1x](\xi) = -i\pi \mbox{sgn}(\xi)$ yield 

\[\int_0^{\infty} B(y) \; dy = \frac{1}{2\pi} \left( - \frac{\pi}{5} + \pi \right) = \frac25.\]

\begin{lemma}[Decay of oscillatory integral $B(x)$, \cite{SS2003}\footnote{In fact, in \cite{SS2003}, the decay of the Airy function is given. However, the same argument can be applied to \eqref{eq:oscil} to obtain Lemma \ref{lem:decay}.}]\label{lem:decay}
	Suppose $x>0$. Then as $x\rightarrow \infty$,
	\begin{itemize}
		\item [(i)] $B(x)\lesssim \langle x \rangle^{-N}$ for all $N>0$.
		\item [(ii)] $B(-x)\lesssim \langle x \rangle^{-3/8}$.
	\end{itemize}
\end{lemma}
We only give a comment on the proof that the proof of (i) follows the repeated integration by parts, while the proof of (ii) essentially follows the van der Corput lemma. Moreover, one can prove $B(x)$ has exponentially decay as $x \to \infty$.

\begin{lemma}[Mellin transform of $B(x)$] \label{mellin}$\;$\\

	 \begin{itemize}
		\item[(i)]For Re $\lambda>0$ we have
		\begin{equation}\label{mellin2}
		\int_0^{\infty}x^{\lambda-1}B(x)dx=\frac{\Gamma(\lambda)\Gamma(\frac15-\frac{\lambda}{5})}{5 \pi}\cos \left(\frac{(1+4\lambda)\pi}{10}\right).
		\end{equation}
\item[(ii)] For $0<\mathrm{Re}\ \lambda <\frac{3}{8}$ we have
	\begin{equation}\label{mellin1}
\int_0^{\infty}x^{\lambda-1}B(-x)dx=\frac{\Gamma(\lambda)\Gamma(\frac15-\frac{\lambda}{5})}{5 \pi}\cos \left(\frac{(1-6\lambda)\pi}{10}\right).
	\end{equation}
		\end{itemize}
\end{lemma}
Remark that a direct calculation gives $\Gamma(\frac15-\frac{\lambda}{5})$ has poles at $\lambda = 1+ 5n$, $n=0,1,2, \cdots$, but
\[\cos \left(\frac{(1+4(1+5n))\pi}{10}\right) = \cos \left(\frac{\pi}{2} + 2n\pi \right) = 0.\]
Moreover, the range of $\mbox{Re} \ \lambda$ relies on the decay rates of $B(x)$ and $B(-x)$ in Lemma \ref{lem:decay}. 
\begin{proof}
The proof basically follows the proof of Lemma 3.3 in \cite{Holmerkdv} and the only difference appears in the change of variable ($\eta = \xi^5$). See \cite{Holmerkdv} for the details.
\end{proof}

\section{Duhamel boundary forcing operator}\label{sec:Duhamel boundary forcing operator}
In this section, we study the Duhamel boundary forcing operator, which was introduced by Colliander and Kenig \cite{CK}, in order to construct the solution to \eqref{eq:5kdv} forced by boundary conditions. We also refer to \cite{Holmerkdv, CC, Cavalcante} for more details.  
\subsection{Duhamel boundary forcing operator class}

We introduce the Duhamel boundary forcing operator associated to the linear Kahawara equation. Let 
\begin{equation}\label{eq:M}
M = \frac{1}{B(0)\Gamma(4/5)}.
\end{equation} 
For $f\in C_0^{\infty}(\mathbb{R}^+)$, define the boundary forcing operator $\mathcal{L}^0$ of order $0$
\begin{equation}\label{eq:BFO}	
\mathcal{L}^0f(t,x):=M\int_0^te^{(t-t')\partial_x^5}\delta_0(x)\mathcal{I}_{-\frac45}f(t')dt'.
\end{equation}
We note that the property of convolution operator ($\partial_x (f * g) = (\partial_xf) * g = f * (\partial_xg)$) and the integration by parts in $t'$ in \eqref{eq:BFO} gives
\begin{equation}\label{eq:BFO1}
\mathcal{L}^0(\partial_t f)(t,x) = M\delta_0(x)\mathcal{I}_{-\frac45}f(t) + \partial_x^5\mathcal{L}^0f(t,x).
\end{equation}
A change of variable and \eqref{eq:oscil} give
\begin{equation}\label{forcing}
\begin{split}
	\mathcal{L}^0f(t,x)&=M\int_0^te^{(t-t')\partial_x^5}\delta_0(x)\mathcal{I}_{-\frac45}f(t')dt'\\
	&=M\int_0^t B\left(\frac{x}{(t-t')^{1/5}}\right)\frac{\mathcal{I}_{-\frac45}f(t')}{(t-t')^{1/5}}dt'.
\end{split}	
\end{equation}

\begin{lemma}\label{continuity}
	Let $f\in C_{0,c}^{\infty}(\mathbb{R}^+)$.
	\begin{itemize}
\item[(a)] For fixed $ 0 \le t \le 1$, $\partial_x^k \mathcal{L}^0f(t,x)$, $k=0,1,2,3$, is continuous in $x \in \mathbb{R}$ and has the decay property in terms of the spatial variable as follows:
\begin{equation}\label{eq:decay1}
|\partial_x^k \mathcal{L}^0f(t,x)| \lesssim_{N} \norm{f}_{H^{N+k}}\bra{x}^{-N}, \quad N \ge 0.
\end{equation}
\item[(b)] For fixed $ 0 \le t \le 1$, $\partial_x^4\mathcal{L}^0f(t,x)$ is continuous in $x$ for $x\neq 0$ and is discontinuous at $x =0$ satisfying
\[\lim_{x\rightarrow 0^{-}}\partial_x^4\mathcal{L}^0f(t,x)=c_1\mathcal{I}_{-4/5}f(t),\ \lim_{x\rightarrow 0^{+}}\partial_x^4\mathcal{L}^0f(t,x)=c_2\mathcal{I}_{-4/5}f(t)\]
for $c_1 \neq c_2$. $\partial_x^4\mathcal{L}^0f(t,x)$ also has the decay property in terms of the spatial variable
\begin{equation}\label{eq:decay2}
|\partial_x^4\mathcal{L}^0f(t,x)| \lesssim_{N} \norm{f}_{H^{N+4}}\bra{x}^{-N}, \quad N \ge 0.
\end{equation}
	\end{itemize}
\end{lemma}

\begin{proof}
The continuity of  $\partial_x^k \mathcal{L}^0f(t,x)$ follows from \eqref{eq:oscil0} with \eqref{B(0)}, \eqref{B'(0)}--\eqref{B'''(0)} for $k=0,1,2,3$ and the proof of \eqref{eq:decay1} exactly follows the proof of Lemma 2.2 in \cite{Holmerkdv}.  Moreover, \eqref{eq:decay1} and \eqref{eq:BFO1} yield that $\partial_x^4\mathcal{L}^0f(t,x)$ is discontinuous only at $x =0$ of size $M\mathcal{I}_{-\frac45}f(t)$ (where $M$ is defined as in \eqref{eq:M}), and has the decay bounds \eqref{eq:decay2}. See \cite{Holmerkdv} for the details.
	\end{proof}
We remark from Lemma \ref{continuity} that $\mathcal{L}^0f(t,0) = f(t)$. We, now, generalize the boundary forcing operator defined as in \eqref{eq:BFO} for $\lambda\in \mathbb{C}$ satisfying $\mbox{Re} \lambda>-5$. 

For $\mbox{Re}\ \lambda > 0$ and given $g\in C_0^{\infty}(\mathbb{R}^+)$, we define
\begin{equation}\label{eq:BFOC}
\mathcal{L}_{\pm}^{\lambda}g(t,x)=\left[\frac{x_{\mp}^{\lambda-1}}{\Gamma(\lambda)}*\mathcal{L}^0\big(\mathcal{I}_{-\frac{\lambda}{5}}g\big)(t, \cdot)   \right](x),
\end{equation}
where $*$ denotes the convolution operator. Note that $\mathcal{L}_{+}^{\lambda}$ (resp. $\mathcal{L}_{-}^{\lambda}$) will be used for the right half-line problem \eqref{kawahararight} (resp. the left half-line problem \eqref{kawaharaleft}). With $\frac{x_{-}^{\lambda-1}}{\Gamma(\lambda)}=\frac{(-x)_+^{\lambda-1}}{\Gamma(\lambda)}$, we also have
 \begin{equation}\label{forcing2}
\mathcal{L}_{+}^{\lambda}g(t,x)=\frac{1}{\Gamma(\lambda)}\int_{x}^{\infty}(y-x)^{\lambda-1}\mathcal{L}^0\big(\mathcal{I}_{-\frac{\lambda}{5}}g\big)(t,y)dy.
\end{equation}
 and 
\begin{equation}\label{forcing1}
\mathcal{L}_{-}^{\lambda}g(t,x)=\frac{1}{\Gamma(\lambda)}\int_{-\infty}^x(x-y)^{\lambda-1}\mathcal{L}^0\big(\mathcal{I}_{-\frac{\lambda}{5}}g\big)(t,y)dy.
\end{equation}
For $\mbox{Re} \lambda > -5$, we use the integration by parts in \eqref{forcing2} and \eqref{forcing1} with the decay property in Lemma \ref{continuity}, and \eqref{eq:BFO1} to obtain
\begin{equation}\label{classe22}
\begin{aligned}
\mathcal{L}_{+}^{\lambda}g(t,x)&=\left[\frac{x_{-}^{(\lambda+5)-1}}{\Gamma(\lambda+5)}*(-\px^5)\mathcal{L}^0\big(\mathcal{I}_{-\frac{\lambda}{5}}g\big)(t, \cdot)   \right](x)\\
&=M\frac{x_{-}^{(\lambda+5)-1}}{\Gamma(\lambda+5)}\mathcal{I}_{-\frac{4}{5}-\frac{\lambda}{5}}g(t)-\int_{x}^{\infty}\frac{(y-x)^{(\lambda+5)-1}}{\Gamma(\lambda+5)}\mathcal{L}^0\big(\partial_t\mathcal{I}_{-\frac{\lambda}{5}}g\big)(t,y)dy
\end{aligned}
\end{equation}
and
\begin{equation}\label{classe12}
\begin{aligned}
\mathcal{L}_{-}^{\lambda}g(t,x)&= \left[\frac{x_{+}^{(\lambda+5)-1}}{\Gamma(\lambda+5)}*\px^5\mathcal{L}^0\big(\mathcal{I}_{-\frac{\lambda}{5}}g\big)(t, \cdot)   \right](x)\\
&=-M\frac{x_{+}^{(\lambda+5)-1}}{\Gamma(\lambda+5)}\mathcal{I}_{-\frac{4}{5}-\frac{\lambda}{5}}g(t)+\int_{-\infty}^x\frac{(x-y)^{(\lambda+5)-1}}{\Gamma(\lambda+5)}\mathcal{L}^0\big(\partial_t\mathcal{I}_{-\frac{\lambda}{5}}g\big)(t,y)dy,
\end{aligned}
\end{equation}
respectively, where $M$ is as in \eqref{eq:M}. We, thus, choose the second term in \eqref{classe22} (resp. \eqref{classe12}) as the definition of $\mathcal{L}_{+}^{\lambda}g(t,x)$ (resp. $\mathcal{L}_{-}^{\lambda}g(t,x)$), and it immediately satisfies (in the sense of distributions)
\begin{equation*}
(\partial_t-\partial_x^5)\mathcal{L}_{-}^{\lambda}g(t,x)=M\frac{x_{+}^{\lambda-1}}{\Gamma(\lambda)}\mathcal{I}_{-\frac{4}{5}-\frac{\lambda}{5}}g(t)
\end{equation*}
and
\begin{equation*}
(\partial_t-\partial_x^5)\mathcal{L}_{+}^{\lambda}g(t,x)=M\frac{x_{-}^{\lambda-1}}{\Gamma(\lambda)}\mathcal{I}_{-\frac{4}{5}-\frac{\lambda}{5}}g(t).
\end{equation*}

\begin{lemma}[Spatial continuity and decay properties for $\mathcal{L}_{\pm}^{\lambda}g(t,x)$]\label{holmer1}
	Let $g\in C_0^{\infty}(\mathbb{R}^+)$ and $M$ be as in \eqref{eq:M}. Then, we have
	\begin{equation}\label{eq:relation}
	\mathcal{L}_{\pm}^{-k}g=\partial_x^k\mathcal{L}^{0}\mathcal{I}_{\frac{k}{5}}g, \qquad k=0,1,2,3,4.
	\end{equation}
	Moreover, $\mathcal{L}_{\pm}^{-4}g(t,x)$ is continuous in $x \in \R \setminus \set{0}$ and has a step discontinuity of size $Mg(t)$ at $x=0$. For $\lambda>-4$, $\mathcal{L}_{\pm}^{\lambda}g(t,x)$ is continuous in $x \in\mathbb{R}$. For $-4\leq\lambda\leq 1$ and  $0\leq t\leq 1$, $\mathcal{L}_{\pm }^{\lambda}g(t,x)$ satisfies the following decay bounds:
	\begin{align*}
	&|\mathcal{L}_{-}^{\lambda}g(t,x)|\leq c_{m,\lambda,g}\langle x\rangle^{-m},\; \text{for all}\quad x\leq 0 \quad \text{and} \quad m\geq0,\\ 
	&|\mathcal{L}_{-}^{\lambda}g(t,x)|\leq c_{\lambda,g}\langle x\rangle^{\lambda-1}, \quad \text{for all} \quad  x\geq 0.\\
	&|\mathcal{L}_{+}^{\lambda}g(t,x)|\leq c_{m,\lambda,g}\langle x\rangle^{-m}, \quad \text{for all} \quad x\geq 0 \quad \text{and} \quad m\geq0,\\ \intertext{and}
	&|\mathcal{L}_{+}^{\lambda}g(t,x)|\leq c_{\lambda,g}\langle x\rangle^{\lambda-1},\quad \text{for all}\quad  x\leq 0.
	\end{align*}
\end{lemma}
\begin{proof}
\eqref{classe22} and \eqref{classe12} immediately show \eqref{eq:relation}. Moreover, Lemma \ref{continuity} guarantees the continuity (except for $x = 0$ when $\lambda = -4$) and discontinuity at $x =0$ of $\mathcal{L}_{\pm}^{\lambda}g$ for $\lambda \ge -4$ and $\lambda = -4$, respectively. The proof of decay bounds can be obtained by using \eqref{classe22}, \eqref{classe12}, \eqref{eq:BFO1} and Lemma \ref{continuity}. For the detailed argument, see Lemma 3.1 in \cite{Holmerkdv}.
\end{proof}

\begin{lemma}[Values of $\mathcal{L}_{+}^{\lambda}f(t,0)$ and $\mathcal{L}_{-}^{\lambda}f(t,0)$ ]\label{trace1}
	 For $\mbox{Re}\ \lambda>-4$,
	\begin{equation}\label{lr0}
	\mathcal{L}_{+}^{\lambda}f(t,0)=\frac{1}{B(0)\Gamma(4/5)}\frac{\cos\left(\frac{(1+4\lambda)\pi}{10}\right)}{5\sin\left(\frac{(1-\lambda)\pi}{5}\right)}f(t)
	\end{equation}
	and
	\begin{equation}\label{ll0}
\mathcal{L}_{-}^{\lambda}f(t,0)= \frac{1}{B(0)\Gamma(4/5)}\frac{\cos\left(\frac{(1-6\lambda)\pi}{10}\right)}{5\sin\left(\frac{(1-\lambda)\pi}{5}\right)}f(t)
	\end{equation}
	\end{lemma}

\begin{proof}
We only prove \eqref{lr0}, since the proof of \eqref{ll0} is similar. We use the formula \eqref{classe22} to have
	\[\mathcal{L}_{+}^{\lambda}f(t,0)
	=-\int_{0}^{\infty}\frac{y^{(\lambda+5)-1}}{\Gamma(\lambda+5)}\mathcal{L}^0\big(\pt \mathcal{I}_{-\frac{\lambda}{5}}f\big)(t,y)dy.\]
It shows	 $	\mathcal{L}_{+}^{\lambda}f(t,0)$ (similarly $	\mathcal{L}_{-}^{\lambda}f(t,0)$ ) is analytic in $\lambda$ for $\mbox{Re}\ \lambda>-4$.

By analyticity argument, it suffices to consider the case when $\lambda$ is a positive real number ($\lambda$ is in $(0,3/8)$ for \eqref{ll0}, in order to use \eqref{mellin1}) and \eqref{forcing} (where $M$ is defined as in \eqref{eq:M}). For the computation we use the representation \eqref{forcing2} for $\lambda > 0$. Fubini Theorem, the change of variable, \eqref{mellin2} and \eqref{eq:IO}, we have 
	\[\begin{split}
	\mathcal{L}_{+}^{\lambda}f(t,0)&=\frac{M}{\Gamma(\lambda)}\int_0^{\infty}y^{\lambda-1}\int_0^t B\left(\frac{y}{(t-t')^{1/5}}\right)\frac{\mathcal{I}_{\frac{-\lambda - 4}{5}}f(t')}{(t-t')^{1/5}}\;dt'dy\\
	&=\frac{M}{\Gamma(\lambda)}\int_0^t (t-t')^{\frac{\lambda + 4}{5}-1}\mathcal{I}_{\frac{-\lambda - 4}{5}}f(t') \int_0^{\infty} y^{\lambda-1}B(y) \; dy dt'\\
	&=\frac{M}{\Gamma(\lambda)}\Gamma\left(\frac{\lambda}{5} + \frac45\right)f(t)\frac{\Gamma(\lambda)\Gamma(\frac15-\frac{\lambda}{5})}{5 \pi}\cos \left(\frac{(1+4\lambda)\pi}{10}\right)\\
	&=\frac{M\cos\left(\frac{(1+4\lambda)\pi}{10}\right)}{5\sin\left(\frac{(1-\lambda)\pi}{5}\right)}f(t).
	\end{split}\]
The last equality is valid thanks to $\Gamma(z)\Gamma(1-z)=\frac{\pi}{\sin\ \pi z}$.
	\end{proof}

\subsection{Linear version}\label{section4}
We consider the linear Kawahara equation
\begin{equation}\label{eq:lin. 5kdv}
\partial_t u -\partial_x^5u =0.
\end{equation}
We define the unitary group associated to \eqref{eq:lin. 5kdv} as
\begin{equation*}
e^{t\partial_x^5}\phi(x)=\frac{1}{2\pi}\int e^{ix\xi}e^{it\xi^5}\hat{\phi}(\xi)d\xi,
\end{equation*}
which allows 
\begin{equation}\label{linear}
\begin{cases}
(\partial_t-\partial_x^5)e^{t\partial_x^5}\phi(x) =0,& (t,x)\in\mathbb{R}\times\mathbb{R},\\
e^{t\partial_x^5}\phi(x)\big|_{t=0}=\phi(x),& x\in\mathbb{R}.
\end{cases}
\end{equation}
Recall $\mathcal{L}_{+}^{\lambda}$ in \eqref{classe22} for the right half-line problem. Let
\[u(t,x)= \mathcal{L}_+^{\lambda_1}\gamma_1(t,x)+\mathcal{L}_+^{\lambda_2}\gamma_2(t,x),\]
where $\gamma_j$ ($j=1,2$) will be chosen later in terms of given boundary data $f$ and $g$.

Let $a_j$ and $b_j$ be constants depending on $\lambda_j$, $j=1,2$, given by
\begin{equation}\label{eq:entries}
a_j = \frac{1}{B(0)\Gamma\left(\frac45\right)}\frac{\cos\left(\frac{(1+4\lambda_j)\pi}{10}\right)}{5\sin\left(\frac{(1-\lambda_j)\pi}{5}\right)}\quad \text{and}\quad b_j = \frac{1}{B(0)\Gamma\left(\frac45\right)}\frac{\cos\left(\frac{(4\lambda_j-3)\pi}{10}\right)}{5\sin\left(\frac{(2-\lambda_j)\pi}{5}\right)}.
\end{equation}
By Lemmas \ref{holmer1} and \ref{trace1}, we get
\begin{equation}\label{eq:f0}
f(t)= u(t,0) = a_1\gamma_1(t)+ a_2\gamma_2(t)
\end{equation}
and
\begin{equation}\label{eq:g0}
g(t)=\partial_x u(t,0) =b_1\mathcal{I}_{-\frac15}\gamma_1(t)+b_2\mathcal{I}_{-\frac15}\gamma_2(t).
\end{equation}
Together with \eqref{eq:f0} and \eqref{eq:g0}, we can write as a matrix form 
\[ \left[\begin{array}{c}
f(t) \\
\mathcal{I}_{\frac15}g(t)  \end{array} \right]=
A
 \left[\begin{array}{c}
\gamma_1(t)\\
\gamma_2(t)   \end{array} \right], \]
where
\[A(\lambda_1,\lambda_2)=
\left[\begin{array}{cc}
a_1 & a_2 \\
	b_1 & b_2 \end{array} \right].\]
We choose an appropriate $\lambda_j$, $j=1,2$, such that $A$ is invertible, and hence $u$ solves
\begin{equation}\label{forçante00}
\begin{cases}
(\partial_t-\partial_x^5)u= 0,\\
u(0,x) = 0,\\
u(t,0) = f(t), \; \px u(t,0) = g(t).
\end{cases}
\end{equation}
Thus, from \eqref{linear} and \eqref{forçante00}, we can construct a solution to linear Kawahara equation \eqref{eq:lin. 5kdv} posed on the right half-line.

Similarly, we can construct  a solution to linear Kawahara equation \eqref{eq:lin. 5kdv} posed on the left half-line by using three boundary conditions, see Section \ref{sec:main proof 2} for more details.

\subsection{Nonlinear version}\label{section4-1}

We define the Duhamel inhomogeneous solution operator $\mathcal{D}$ as
\begin{equation*}
\mathcal{D}w(t,x)=\int_0^te^{(t-t')\partial_x^5}w(t',x)dt',
\end{equation*}
it follows that
\[\left \{
\begin{array}{l}
(\partial_t-\partial_x^5)\mathcal{D}w(t,x) =w(t,x),\ (t,x)\in\mathbb{R}\times\mathbb{R},\\
\mathcal{D}w(x,0) =0,\ x\in\mathbb{R}.
\end{array}
\right.\]

Let 
\[u(t,x)= \mathcal{L}_+^{\lambda_1}\gamma_1(t,x)+\mathcal{L}_+^{\lambda_2}\gamma_2(t,x)+ e^{t\px^5}\phi(x) + \mathcal{D}w.\]
Similarly as in Subsection \ref{section4}, but $\gamma_1$ and $\gamma_2$ are dependent on not only $f$ and $g$, but also $e^{t\px^5}\phi(x)$ and $\mathcal{D}w$, we see that $u$ solves
\[\begin{cases}
(\partial_t-\partial_x^5)u= w,\\
u(0,x) = \phi(x),\\
u(t,0) = f(t), \; \px u(t,0) = g(t).
\end{cases}\]
Similarly, we can construct a solution to nonlinear Kawahara equation posed on the left half-line. See Sections  \ref{sec:main proof 1} and \ref{sec:main proof 2} for more details.

\section{Energy estimates}\label{sec:energy}
In this section, we are going to prove the energy estimate in the solution space defined as in Subsection \ref{sec:sol space}.

\begin{lemma}\label{grupo}
	Let $s\in\mathbb{R}$ and $b, \alpha \in \R$. If $\phi\in H^s(\mathbb{R})$, then
	\begin{itemize}
		\item[(a)] \emph{(Space traces)} $\|\psi(t)e^{t\partial_x^5}\phi(x)\|_{C\big(\mathbb{R}_t;\,H_x^s(\mathbb{R})\big)}\lesssim \|\phi\|_{H^s(\mathbb{R})}$;
		\item[(b)] \emph{((Derivatives) Time traces)} 
		\[
		\|\psi(t) \partial_x^{j}e^{t\partial_x^5}\phi(x)\|_{C_x(\mathbb{R};H_t^{\frac{s+2-j}{5}}(\mathbb{R}))}\lesssim_{\psi, s, j} \|\phi\|_{H^s(\mathbb{R})}, \quad j\in\{0,1,2\};
		\]
		\item [(c)] \emph{(Bourgain spaces)} $\|\psi(t)e^{t\partial_x^5}\phi(x)\|_{X^{s,b}\cap D^{\alpha}}\lesssim_{\psi, b, \alpha}  \|\phi\|_{H^s(\mathbb{R})}$.
	\end{itemize}
\end{lemma}
\begin{proof}
The proofs of (a) and (c) are standard and the proof of (b) follows from the smoothness of $\psi$ and the local smoothing estimate \eqref{eq:lsm}. We omit the details.
\end{proof}

\begin{lemma}\label{duhamel}
Let $s\in \mathbb{R}$. For $0 < b < \frac12 < \alpha < 1-b$, we have
	\begin{itemize}
		\item[(a)] \emph{(Space traces)} 
		\begin{equation*}
		\|\psi(t)\mathcal{D}w(t,x)\|_{C\big(\mathbb{R}_t;\,H^s(\mathbb{R}_x)\big)}\lesssim \|w\|_{X^{s,-b}};
		\end{equation*}
		\item[(b)] \emph{((Derivatives) Time traces)}
		\begin{equation*}
		\|\psi(t) \partial_x^j\mathcal{D}w(t,x)\|_{C(\mathbb{R}_x;H^{\frac{s+2-j}{5}}(\mathbb{R}_t))}\lesssim	\|w\|_{X^{s,-b}}+\|w\|_{Y^{s,-b}};
		\end{equation*}
		\item[(c)] \emph{(Bourgain spaces estimates)} 
		\begin{equation*}
		\|\psi(t)\mathcal{D}w(t,x)\|_{X^{s,b} \cap D^{\alpha}}\lesssim  \|w\|_{X^{s,-b}}.
		\end{equation*}
	\end{itemize}
\end{lemma}
Remark in (b) that $\|\psi(t) \partial_x^j\mathcal{D}w(t,x)\|_{C(\mathbb{R}_x;H_0^{\frac{s+2-j}{5}}(\mathbb{R}_t^+))}$ has same bound for $s < \frac{11}{2} + j$. 

\begin{proof}
The idea of the proof of Lemma \ref{duhamel} follows Section 5 in \cite{CK}. Here we give the details for the sake of reader's convenience. 

\textbf{(a).} A direct calculate gives
\begin{equation}\label{eq:duhamel fourier(a)}
\psi(t)\mathcal{D}w(t,x) = c\int e^{ix\xi}e^{it\xi^5}\psi(t) \int \wt{w}(\tau',\xi) \frac{e^{it(\tau'-\xi^5) }-1}{i(\tau' - \xi^5)} \; d\tau'd\xi.
\end{equation}
We denote by $w = w_1 + w_2$, where 
\[\wt{w}_1(\tau,\xi) = \eta_0(\tau-\xi^5)\wt{w}(\tau,\xi),\]
where $\eta_0$ is defined in Subsection \ref{sec:sol space}.

For $w_1$, we use the  Taylor expansion of $e^x$ at $x =0$. Then, we can rewrite \eqref{eq:duhamel fourier(a)} for $w_1$ as
\[
\begin{aligned}
\psi(t)\mathcal{D}w(t,x) &= c\int e^{ix\xi}e^{it\xi^5}\psi(t) \int \wt{w}_1(\tau',\xi) \frac{e^{it(\tau'-\xi^5) }- 1}{i(\tau' - \xi^5)} \; d\tau'd\xi\\
&=c\sum_{k=1}^{\infty}\frac{i^{k-1}}{k!}\psi^k(t) \int e^{ix\xi}e^{it\xi^5}\wh{F}_1^k(\xi)\; d\xi\\
&=c\sum_{k=1}^{\infty}\frac{i^{k-1}}{k!}\psi^k(t)e^{t\partial_x^5}F_1^k(x),
\end{aligned}
\]
where $\psi^k(t) = t^k\psi(t)$ and 
\begin{equation}\label{eq:linear estimate0}
\wh{F}_1^k(\xi) = \int \wt{w}_1(\tau,\xi) (\tau-\xi^5)^{k-1} \; d\tau.
\end{equation}
Since
\begin{equation}\label{eq:linear estimate}
\norm{F_1^k}_{H^s} = \left(\int \bra{\xi}^{2s} \left| \int \wt{w}_1(\tau,\xi) (\tau-\xi^5)^{k-1} \; d\tau \right|^2 \; d\xi \right)^{\frac12} \lesssim \norm{w}_{X^{s,-b}},
\end{equation}
we have from Lemma \ref{grupo} (a) that
\[
\norm{\psi(t) \mathcal{D}w(t,x)}_{C_tH^s} \lesssim \sum_{k=1}^{\infty}\frac{1}{k!}\norm{F_1^k}_{H_x^s} \lesssim \norm{w}_{X^{s,-b}}.
\]
For $w_2$, a direct calculation gives
\begin{equation}\label{eq:duhamel fourier}
\ft[\psi\mathcal{D}w](\tau,\xi) = c\int \wt{w}_2(\tau',\xi) \frac{\wh{\psi}(\tau-\tau') - \wh{\psi}(\tau - \xi^5)}{i(\tau' - \xi^5)} \; d\tau'.
\end{equation}
Since $\norm{\psi \mathcal{D}w}_{C_tH^s} \lesssim \norm{\bra{\xi}^s\ft[\psi\mathcal{D}w](\tau,\xi)}_{L_{\xi}^2L_{\tau}^1}$, it suffices to control
\begin{equation}\label{eq:a.1}
\left(\int \bra{\xi}^{2s}\left|\int |\wt{w}_2(\tau',\xi)| \int \frac{|\wh{\psi}(\tau-\tau') - \wh{\psi}(\tau - \xi^5)|}{|\tau' - \xi^5|} \; d\tau d\tau' \right|^2 \; d\xi \right)^{\frac12},
\end{equation}
due to \eqref{eq:duhamel fourier}. We use the $L^1$ integrability of $\wh{\psi}$, so that
\[
\eqref{eq:a.1} \lesssim \left(\int \bra{\xi}^{2s}\left|\int_{|\tau'-\xi^5| > 1} \frac{|\wt{w}_2(\tau',\xi)|}{|\tau' - \xi^5|} d\tau' \right|^2 \; d\xi \right)^{\frac12} \lesssim \norm{w}_{X^{s,-b}}.
\]

\textbf{(b).}  Similarly as \eqref{eq:duhamel fourier(a)}, a direct calculate gives
\begin{equation}\label{eq:duhamel fourier(b)}
\psi(t)\partial_x^j\mathcal{D}w(t,x) = c\int e^{ix\xi}e^{it\xi^5}(i\xi)^j\psi(t) \int \wt{w}(\tau',\xi) \frac{e^{it(\tau'-\xi^5) }-1}{i(\tau' - \xi^5)} \; d\tau'd\xi.
\end{equation}
We split $w = w_1 + w_2$ as in the proof of (a). For $w_1$, similarly, but we use Lemma \ref{grupo} (b) to obtain that
\[
\norm{\psi\partial_x^j\mathcal{D}w}_{L_x^{\infty}H_t^{\frac{s+2-j}{5}}} \lesssim \norm{w}_{X^{s,-b}}.
\]

For $w_2$, recall \eqref{eq:duhamel fourier(b)}
\[
\begin{aligned}
\psi\partial_x^j\mathcal{D}w(t,x) &= c\int e^{ix\xi}e^{it\xi^5}(i\xi)^j\psi(t) \int \frac{\wt{w}(\tau',\xi)}{i(\tau' - \xi^5)} \left(e^{it(\tau'-\xi^5) }- 1\right) \; d\tau'd\xi\\
&= I -II.
\end{aligned}
\]
We first consider $II$. Let 
\[\wh{W}(\xi) = \int \frac{\wt{w}_2(\tau,\xi)}{i(\tau-\xi^5)} \; d\tau.\]
Note that 
\begin{equation}\label{eq:linear estimate2}
\norm{W}_{H^s} \lesssim \norm{w_2}_{X^{s,-b}},
\end{equation}
for $b < \frac12$. Then, it immediately follows from 
\[II = \psi(t) \partial_x^j e^{t\partial_x^5}W(x)\]
and Lemma \ref{grupo} (b) that
\[\norm{\psi \partial_x^j e^{t\partial_x^5}W}_{C_xH_t^{\frac{s+2-j}{5}}} \lesssim \norm{W}_{H^s} \lesssim \norm{w}_{X^{s,-b}}.\]

Now it remains to deal with $I$. Taking the Fourier transform to $I$ with respect to $t$ variable, we have
\[
\int e^{ix\xi} (i\xi)^j \int \frac{\wt{w}_2(\tau',\xi)}{i(\tau'-\xi^5)} \wh{\psi}(\tau-\tau')\; d\tau'd\xi,
\]
and hence it suffices to control
\begin{equation}\label{eq:b.8}
\left( \int \bra{\tau}^{\frac{2(s+2-j)}{5}} \left| \int e^{ix\xi} (i\xi)^j \int \frac{\wt{w}_2(\tau',\xi)}{i(\tau'-\xi^5)} \wh{\psi}(\tau-\tau')\; d\tau'd\xi \right|^2 \; d\tau \right)^{\frac12}.
\end{equation} 
We first split the region in $\tau$ as follows:
\[\mathbf{Case\; I. }\; |\tau| \le 1, \qquad \mathbf{Case\; II. }\; 1 < |\tau|.\]

\textbf{Case I.} $|\tau| \le 1$. In this case, the weight $\bra{\tau}^{\frac{2(s+2-j)}{5}}$ and the integration with respect to $\tau$ can be negligible. If $|\xi|^5 \le 1$, the weight $|\xi|^{j}$ and the integration with respect to $\xi$ are negligible as well. Then, the Cauchy-Schwarz inequality gives
\[\eqref{eq:b.8} \lesssim \norm{w_2}_{X^{s,-b}},\]
for $b < \frac12$. When $1<|\xi|^5$, we split the region in $\tau'$ similarly as 
\[\mathbf{I. }\; |\tau'-\xi^5| < \frac12|\xi|^5, \qquad \mathbf{II. }\; 2|\xi|^5 < |\tau'-\xi^5|, \qquad \mathbf{III. }\; \frac12|\xi|^5 \le |\tau' - \xi^5| \le 2|\xi|^5.\]
For \textbf{I} ($|\tau'-\xi^5| < \frac12|\xi|^5$), we can perform the integration with respect to $\xi$ in \eqref{eq:b.8} by using $|\wh{\psi}(\tau - \tau')| \lesssim |\xi|^{-5k}$, while we use $|\wh{\psi}(\tau - \tau')| \lesssim |\tau' - \xi^5|^{-k}$ in the case \textbf{II} ($2|\xi|^5 < |\tau'-\xi^5|$) for large $k \gg 1$. Hence, we have for both cases that
\[\eqref{eq:b.8} \lesssim \norm{w_2}_{X^{s,-b}}.\]

The case \textbf{III} ($\frac12|\xi|^5 < |\tau'-\xi^5| < 2|\xi|^5$) is more complicated. Since $|\tau| \le 1$, this case is equivalent to the case when $\frac12|\tau-\xi^5| < |\tau'-\xi^5| < 2|\tau-\xi^5|$. If $(\tau'-\xi^5)\cdot(\tau-\xi^5)<0$, by using the facts that $|\wh{\psi}(\tau - \tau')| \lesssim |\xi|^{-5k}$ and
\[ \int_{|\tau'-\xi^5| \sim |\xi|^5} 1 \; d\tau' \lesssim |\xi|^{5},\]
we obtain
\[\eqref{eq:b.8} \lesssim \norm{w_2}_{X^{s,-b}}.\]

Otherwise ($(\tau'-\xi^5)\cdot(\tau-\xi^5) > 0$), we further divide the case into $|(\tau-\xi^5) - (\tau'-\xi^5)| < 1$ and $|(\tau-\xi^5) - (\tau'-\xi^5)| > 1$. For the former case, let 
\begin{equation}\label{eq:b.6}
\Phi(\tau',\xi) = |\xi|^s|\tau'-\xi^5|^{-b}\wt{w}_2(\tau',\xi).
\end{equation}
We note that $\norm{\Phi}_{L^2} \sim \norm{w_2}_{X^{s,-b}}$. Since $|\wh{\psi}(\tau - \tau') | \lesssim 1$, we have
\[
\left|\int_{|\tau' - \tau| < 1}|\xi|^s|\tau'-\xi^5|^{-b}\wt{w}_2(\tau',\xi)\wh{\psi}(\tau' - \tau) \; d\tau' \right| \lesssim \left|\int_{|\eta - \tau| < 1} \Phi(\tau',\xi) \; d\tau' \right| \lesssim M\Phi(\tau,\xi),
\]
where $Mf(x)$ is the Hardy-Littlewood maximal function of $f$. Note that $\norm{Mf}_{L^p} \lesssim \norm{f}_{L^p}$ for $1< p \le \infty$ (see, in particular, \cite{Stein1993}).
Since
\begin{equation}\label{eq:restriction}
\int_{|\xi| > 1} |\xi|^{2(j-s-5+5b)} \; d\xi < \infty, \qquad \frac{s+2-j}{5} > 0,
\end{equation}
we have
\[
\eqref{eq:b.8} \lesssim \left(\int |M\Phi(\tau, \xi)|^2 \; d\tau d\xi \right)^{\frac12} \lesssim \norm{\Phi}_{L^2} \lesssim \norm{w_2}_{X^{s,-b}}.
\]
For the latter case, the integration region in $\tau'$ can be reduced to $\tau-\xi^5 + 1 < \tau' -\xi^5 < 2(\tau -\xi^5)$ for positive $\tau - \xi^5$ and $\tau'-\xi^5$, since the exact same argument can be applied to the other regions.\footnote{Indeed, we have only four regions; $\tau-\xi^5 + 1 < \tau'-\xi^5 < 2(\tau-\xi^5)$ and $\frac12 (\tau-\xi^5) < \tau'-\xi^5 < \tau-\xi^5 - 1$ for positive $\tau-\xi^5, \tau'-\xi^5$, and $\tau-\xi^5 + 1 < \tau'-\xi^5 < \frac12(\tau-\xi^5)$ and $2(\tau-\xi^5)  < \tau'-\xi^5 < \tau-\xi^5 - 1$ for negative $\tau-\xi^5, \tau'-\xi^5$, and the same argument can be applied on each region. \label{fn:4regions}} 
Since $|\wh{\psi}(\tau - \eta) | \lesssim |\tau - \tau'|^{-k}$ in this case, the left-hand side of \eqref{eq:b.8} is bounded by
\begin{equation}\label{eq:b.7}
\left(\int_{|\tau| < 1} \left| \int_{|\xi|^5 > 1} |\xi|^{j-s-5+5b}\int_{\tau + 1}^{2\tau} \frac{\Phi(\tau',\xi)}{|\tau-\tau'|^k} \; d\tau' d\xi \right|^2 \; d\tau \right)^{\frac12},
\end{equation}
where $\Phi$ is defined as in \eqref{eq:b.6}. Let $\epsilon = (k-1)/2$ for $k > 1$. Then, the change of variable, the Cauchy-Schwarz inequality, \eqref{eq:restriction} and the Fubini theorem yield
\[\begin{aligned}
\eqref{eq:b.7} &\lesssim \left(\int_{|\tau| < 1} \left| \int_{|\xi|^5 > 1} |\xi|^{j-s-5+5b}\int_{1}^{\tau} \frac{\Phi(\tau + h,\xi)}{|h|^k} \; dh d\xi \right|^2 \; d\tau \right)^{\frac12}\\
&\lesssim \left(\int_{|\tau| < 1}\int_{|\xi|^5>1}\int_{|h| > 1} \frac{|\Phi(\tau + h,\xi)|^2}{|h|^{2k-1-2\epsilon}} \; dh d\xi d\tau \right)^{\frac12}\\
&\lesssim \norm{\Phi}_{L^2} \lesssim \norm{w_2}_{X^{s,-b}}.
\end{aligned}\]
Remark from \eqref{eq:restriction} that it is essential to introduce $Y^{s,b}$ to cover whole negative range of regularity as mentioned in Subsection \ref{sec:sol space}. We use $Y^{s,b}$ space for the case when $\frac12|\xi|^5 < |\tau'-\xi^5| < 2|\xi|^5$. It suffice to consider
\begin{equation}\label{eq:b.8-1}
\left( \int_{|\tau| \le 1}  \left| \int_{|\xi|^5\le 1} |\xi|^j \int_{|\tau' - \xi^5| > 1} \frac{\wt{w}_2(\tau',\xi)}{|\tau'-\xi^5|} \wh{\psi}(\tau-\tau')\; d\tau'd\xi \right|^2 \; d\tau \right)^{\frac12}.
\end{equation}

We may assume $|\tau'| > 2$, otherwise, we use $\bra{\tau'}^{\frac s5} \sim 1$ and $|\wh{\psi}(\tau-\tau')| \lesssim 1$ to obtain 
\[\eqref{eq:b.8-1} \lesssim \norm{w_2}_{Y^{s,-b}}.\]
Since $\bra{\tau'} \sim \bra{\tau-\tau'}$, we have  
\begin{equation}\label{eq:Ysb}
\begin{aligned}
\eqref{eq:b.8-1} &\lesssim \left| \int_{|\xi|^5 > 1} |\xi|^{j-5+5b} \int_{|\tau'|> 2} \bra{\tau'}^{\frac s5}\bra{-b}\wt{w}_2(\tau',\xi)\bra{\tau-\tau'}^{-\frac s5}\wh{\psi}(\tau-\tau')\; d\tau'd\xi \right|\\
&\lesssim \left( \int_{|\tau| \ge 1} |\tau|^{-\frac{2s}{5}}|\wh{\psi}(\tau)|^2 \; d\tau \right)^{\frac12} \norm{w_2}_{Y^{s,-b}} \lesssim \norm{w_2}_{Y^{s,-b}}.
\end{aligned}
\end{equation}
Thus we cover whole regularity $s \in \R$.

\textbf{Case II.} $1 < |\tau|$. This case is much more complicated. When $|\xi|^5 < 1$, $\xi^5$ is negligible, and hence \eqref{eq:b.8} is reduced to
\[\left( \int_{|\tau| > 1} |\tau|^{\frac{2(s+2-j)}{5}} \left|  \int_{|\tau'| > 1} \frac{\wt{w}_2^{\ast}(\tau')}{|\tau'|} \wh{\psi}(\tau-\tau')\; d\tau' \right|^2 \; d\tau \right)^{\frac12},\]
where $\wt{w}_2^{\ast}(\tau') = \norm{\bra{\cdot}^s\wt{w}_2(\cdot,\tau')}_{L^2}$. Then, the following cases can be treated via the similar way:
\begin{itemize}
\item[II.a] $|\tau'| < \frac12 |\tau|$, in this case we use $|\wh{\psi}(\tau-\tau')| \lesssim |\tau|^{-k}$,
\item[II.b] $2|\tau| < |\tau'|$, in this case we use $|\wh{\psi}(\tau-\tau')| \lesssim |\tau'|^{-k}$,
\item[II.c] $\frac12 |\tau| < |\tau'| < 2 |\tau|$ and $\tau \cdot \tau' < 0$, in this case we use $|\wh{\psi}(\tau-\tau')| \lesssim |\tau|^{-k}$.
\end{itemize}
For the case when $\tau \cdot \tau' > 0$, we, similarly, split the case into $|\tau - \tau'| < 1$ and $|\tau-\tau'|>1$. Then, by using the Hardy-Littlewood maximal function of $|\tau'|^{-b}\wt{w}_2^{\ast}(\tau')$ for $|\tau - \tau'| < 1$, and the smoothness of $\psi$ ($|\wh{\psi}(\tau - \tau')| \lesssim |\tau-\tau'|^{-k}$) for $1 < |\tau - \tau'| < |\tau|$ similarly as before, we have for the rest case that
\[\eqref{eq:b.8} \lesssim \norm{w_2}_{X^{s,-b}}, \qquad (s+2-j)/5 \le \frac12\]
The similar argument as \eqref{eq:Ysb} yields 
\[\eqref{eq:b.8} \lesssim \norm{w_2}_{Y^{s,-b}},\]
when $(s+2-j)/5 > 1/2$.

Now we consider the case when $|\xi|^5 > 1$. For given $\tau, \xi$, we further divide the case into $|\tau' - \xi^5| \le \frac12 |\tau - \xi^5|$, $2|\tau - \xi^5| \le |\tau' - \xi^5|$ and $\frac12|\tau - \xi^5| < |\tau' - \xi^5| < 2|\tau - \xi^5|$.

For the case when $|\tau' - \xi^5| \le \frac12 |\tau - \xi^5|$, we know $|\tau - \xi^5| > 1$ and $|\wh{\psi}(\tau - \tau')| \lesssim |\tau-\xi^5|^{-k}$. Moreover, the region of $\xi$ can be expressed as $\cup_{j=1}^{4}\mathcal{A}_j$, where
\[\mathcal{A}_1 = \left\{\xi : |\xi|^5 > 1, \; 2|\tau| < |\xi|^5 \right\},\] 
\[\mathcal{A}_2 = \left\{\xi : |\xi|^5 > 1, \; |\xi|^5 < \frac12|\tau| \right\},\]
\[\mathcal{A}_3 = \left\{\xi : |\xi|^5 > 1, \; \frac12|\tau| \le |\xi|^5 \le 2|\tau|, \; \tau \cdot \xi^5 < 0 \right\}\]
and 
\[\mathcal{A}_4 = \left\{\xi : |\xi|^5 > 1, \; \frac12|\tau| \le |\xi|^5 \le 2|\tau|, \; \tau \cdot \xi^5 > 0 \right\}.\]
On $\mathcal{A}_1$, we have $|\tau|^{\frac{s+2-j}{5}} \lesssim |\xi|^{s+2-j}$\footnote{This property restricts the regularity condition as $\frac{s+2-j}{5} > 0$. However, in the case when $\frac{s+2-j}{5} \le 0$, since $|\tau|^{\frac{s+2-j}{5}} \lesssim 1$, the same argument yields
\[\eqref{eq:b.8} \lesssim \norm{w_2}_{X^{s,-b}}.\]}, $|\tau-\xi^5| \sim |\xi|^5$ and $|\wh{\psi}(\tau-\tau')| \lesssim |\xi|^{-5k}$ for $k > 1$. Then, we have
\begin{equation}\label{eq:b.10}
\begin{aligned}
\eqref{eq:b.8} &\lesssim \left( \int_{|\tau| > 1}  \left| \int_{\mathcal{A}_1} |\xi|^{2-5k} \int_{1 < |\tau'-\xi^5| } |\tau'-\xi^5|^{-1+b}\wt{\Phi}(\tau',\xi)\; d\tau'd\xi \right|^2 \; d\tau \right)^{\frac12}\\
&\lesssim \left( \int_{|\tau| > 1}  |\tau|^{1-2k} \; d\tau \right)^{\frac12} \norm{\Phi}_{L^2} \lesssim \norm{w_2}_{X^{s,-b}},
\end{aligned}
\end{equation}
where $\Phi$ is defined as in \eqref{eq:b.6}. On $\mathcal{A}_2$, we have $|\tau-\xi^5| \sim |\tau|$ and $|\wh{\psi}(\tau-\tau')| \lesssim |\tau|^{-k}$ for $k > 1$. Then, similarly as \eqref{eq:b.10}, we have
\[\eqref{eq:b.8} \lesssim \norm{w_2}_{X^{s,-b}}.\]
On $\mathcal{A}_3$, since $|\tau-\xi^5| \sim |\tau| \sim |\xi|^5$, we have 
\[\eqref{eq:b.8} \lesssim \norm{w_2}_{X^{s,-b}},\] 
similarly as on $\mathcal{A}_1$ or $\mathcal{A}_2$. On $\mathcal{A}_4$, we have $|\tau|^{\frac{s+2-j}{5}} \sim |\xi|^{s+2-j}$ and $|\wh{\psi}(\tau-\tau')| \lesssim |\tau-\xi^5|^{-k}$ for $k > 1$. Moreover, it is enough to consider the region $\tau +1 < \xi^5 < 2\tau$ due to the footnote \ref{fn:4regions}. Then, we have
\begin{align}
\eqref{eq:b.8} &\lesssim \left( \int_{\tau > 1}  \left| \int_{\tau + 1}^{2\tau} \xi^{2}|\tau-\xi^5|^{-k} \int_{1 < |\tau'-\xi^5| } |\tau'-\xi^5|^{-1+b}\wt{\Phi}(\tau',\xi)\; d\tau'd\xi \right|^2 \; d\tau \right)^{\frac12} \nonumber \\
&\lesssim \left( \int_{|\tau| > 1}  \left| \int_{\tau + 1}^{2\tau} \xi^{2}|\tau-\xi^5|^{-k}\wt{\Phi}^{\ast}(\xi)\; d\xi \right|^2 \; d\tau \right)^{\frac12}, \label{eq:b.11}
\end{align}
where $\wt{\Phi}^{\ast}(\xi) = \norm{\wt{\Phi}(\cdot, \xi)}_{L^2}$. Let $h =  \xi^5 - \tau$. Then the change of variables, the Cauchy-Schwarz inequality and the Fubini theorem yields
\[\begin{aligned}
\eqref{eq:b.11} &\lesssim \left( \int_{|\tau| > 1}  \left| \int_{1}^{\tau} |h|^{-k}\wt{\Phi}^{\ast}((\tau + h)^{\frac15})(\tau + h)^{-\frac25}\; dh \right|^2 \; d\tau \right)^{\frac12}\\
&\lesssim \left( \int_{|\tau| > 1} \int_{1}^{\tau} |h|^{-2k+1+2\epsilon}|\wt{\Phi}^{\ast}((\tau + h)^{\frac15})|^2(\tau + h)^{-\frac45}\; dh  \; d\tau \right)^{\frac12}\\
&\lesssim \norm{w_2}_{X^{s,-b}},
\end{aligned}\]
for small $0<\epsilon \ll 1$, which implies 
\[\eqref{eq:b.8} \lesssim \norm{w_2}_{X^{s,-b}}.\]
For the case when $2|\tau - \xi^5| \le |\tau' - \xi^5|$, the region of $\xi$ can be further divided by 
\[\mathcal{B}_1 = \left\{\xi : |\xi|^5 > 1, \; |\tau - \xi^5| < 1 \right\} \quad \mbox{and} \quad\mathcal{B}_2 = \left\{\xi : |\xi|^5 > 1, \; |\tau - \xi^5|  \ge 1  \right\}.\]
On $\mathcal{B}_1$, we know $|\tau|^{\frac{s+2-j}{5}} \sim |\xi|^{s+2-j}$. Since $|\wh{\psi}(\tau-\tau')| \lesssim 1$ and 
\[\int_{1 < |\tau'-\xi^5| } |\tau'-\xi^5|^{-1+b}\wt{\Phi}(\tau',\xi)\; d\tau' \lesssim \wt{\Phi}^{\ast}(\xi),\]
we have from the change of variable ($\eta = \xi^5$) that
\[
\begin{aligned}
\eqref{eq:b.8} &\lesssim  \left( \int_{|\tau| > 1}  \left| \int_{|\eta - \tau| < 1} \wt{\Phi}^{\ast}(\eta^{\frac15})\eta^{-\frac25}\; d\eta \right|^2 \; d\tau \right)^{\frac12}\\
&\lesssim  \left( \int_{|\tau| > 1}  |M\wt{\Phi}^{\ast\ast}(\tau)|^2 \; d\tau \right)^{\frac12},
\end{aligned}
\]
where $\wt{\Phi}^{\ast\ast}(\eta) = \wt{\Phi}^{\ast}(\eta^{\frac15})\eta^{-\frac25}$. Note that $\norm{\wt{\Phi}^{\ast\ast}}_{L^2} = c\norm{w_2}_{X^{s,-b}}$. Therefore, we have
\[\eqref{eq:b.8} \lesssim \norm{w_2}_{X^{s,-b}}.\]
On $\mathcal{B}_2$, by dividing the region of $\xi$ as $\mathcal{A}_j$, $j=1,2,3,4$, we have similarly 
\[\eqref{eq:b.8} \lesssim \norm{w_2}_{X^{s,-b}}.\]

For the rest case ($\frac12|\tau - \xi^5| < |\tau' - \xi^5| < 2|\tau - \xi^5|$), we further divide the region of $\tau'$ as $ \mathcal{C}_1\cup \mathcal{C}_2$, where
\[\mathcal{C}_1 = \left\{\tau' : |\tau'| > 1, \;  \frac12|\tau - \xi^5| < |\tau' - \xi^5| < 2|\tau - \xi^5|, \; (\tau'-\xi^5)\cdot(\tau-\xi^5)<0\right\}\]
and
\[\mathcal{C}_2 = \left\{\tau' : |\tau'| > 1, \;  \frac12|\tau - \xi^5| < |\tau' - \xi^5| < 2|\tau - \xi^5|, \; (\tau'-\xi^5)\cdot(\tau-\xi^5)>0\right\}.\]
On $\mathcal{C}_1$, since
\[|\wh{\psi}(\tau - \tau')| \lesssim |\tau-\xi^5|^{-k} \sim |\tau-\xi^5|^{-k},\]
for $k \ge 0$, by dividing the region of $\xi$ as $\mathcal{A}_j$, $j=1,2,3,4$, we have similarly 
\[\eqref{eq:b.8} \lesssim \norm{w_2}_{X^{s,-b}}.\]
On the other hand, we further split the set $\mathcal{C}_2$ by
\[\mathcal{C}_{21} = \left\{\tau' : |\tau'| > 1, \;  \frac12|\tau - \xi^5| < |\tau' - \xi^5| < 2|\tau - \xi^5|, \; (\tau'-\xi^5)\cdot(\tau-\xi^5)>0,\; |\tau - \tau'| < 1 \right\}\] 
and
\[\mathcal{C}_{22} = \left\{\tau' : |\tau'| > 1, \;  \frac12|\tau - \xi^5| < |\tau' - \xi^5| < 2|\tau - \xi^5|, \; (\tau'-\xi^5)\cdot(\tau-\xi^5)>0 ,\; |\tau - \tau'| > 1 \right\}.\] 
On $\mathcal{C}_{21}$, \eqref{eq:b.8} is reduced by
\begin{equation}\label{eq:b.12}
\left( \int_{|\tau| > 1} |\tau|^{\frac{2(s+2-j)}{5}} \left| \int_{|\xi|^5>1} |\xi|^{j-s}|\tau-\xi^5|^{-1+b}M\wt{\Phi}(\tau,\xi)d\xi \right|^2 \; d\tau \right)^{\frac12}.
\end{equation}
Then, by dividing the region of $\xi$ in \eqref{eq:b.12} as $\mathcal{A}_{j}$, $j=1,2,3,4$, we have similarly 
\[\eqref{eq:b.8} \lesssim \norm{w_2}_{X^{s,-b}}.\]
On $\mathcal{C}_{22}$, we know $|\wh{\psi}(\tau-\tau')| \lesssim |\tau-\tau'|^{-k}$ for $k \ge 0$. Then, \eqref{eq:b.8} is reduced by
\begin{equation}\label{eq:b.13}
\begin{aligned}
\Big( \int_{|\tau| > 1} |\tau|^{\frac{2(s+2-j)}{5}}  \Big| \int_{|\xi|^5>1} |\xi|^{j-s}&|\tau-\xi^5|^{-1+b}\\
&\times\Big(\int_{1}^{\tau-\xi^5}|h|^{-2k+1+2\epsilon}|\wt{\Phi}(\tau+h,\xi)|^2 \; dh\Big)^{\frac12} d\xi \Big|^2 \; d\tau \Big)^{\frac12},
\end{aligned}
\end{equation}
for small $0< \epsilon \ll 1$. Then, for $k \gg 1$ large enough, by dividing the region of $\xi$ in \eqref{eq:b.13} as $\mathcal{A}_{j}$, $j=1,2,3,4$, we have similarly 
\[\eqref{eq:b.8} \lesssim \norm{w_2}_{X^{s,-b}}.\]
Therefore, we have
\[\|\psi(t) \partial_x^j\mathcal{D}w(t,x)\|_{C(\mathbb{R}_x;H^{\frac{s+2-j}{5}}(\mathbb{R}_t))} \lesssim  \|w\|_{X^{s,-b}} + \|w\|_{Y^{s,-b}} .\]

\textbf{(c).} We split $w=w_1+w_2$ similarly as before. For $w_1$, by Lemma \ref{grupo} (c) and \eqref{eq:linear estimate}, we have
\[\norm{\psi(t) \mathcal{D}w_1(t,x)}_{X^{s,b} \cap D^{\alpha}} \lesssim \sum_{k=1}^{\infty}\frac{1}{k!}\norm{F_1^k}_{H_x^s} \lesssim \norm{w}_{X^{s,-b}},\]
where $F_1^k$ is defined as in \eqref{eq:linear estimate0}.
 
For $w_2$, recall \eqref{eq:duhamel fourier(b)}
\[
\begin{aligned}
\psi\partial_x^j\mathcal{D}w(t,x) &= c\int e^{ix\xi}e^{it\xi^5}(i\xi)^j\psi(t) \int \frac{\wt{w}(\tau',\xi)}{i(\tau' - \xi^5)} \left(e^{it(\tau'-\xi^5) }- 1\right) \; d\tau'd\xi\\
&= I -II.
\end{aligned}
\]
Then, we use Lemma \ref{grupo} (c) and \eqref{eq:linear estimate2} for $II$ to obtain
\[\norm{\psi  e^{t\partial_x^5}W}_{X^{s,b} \cap D^{\alpha}} \lesssim \norm{W}_{H^s} \lesssim \norm{w}_{X^{s,-b}},\]
where $W$ is defined as in \eqref{eq:linear estimate2}.

Now it remains to show
\begin{equation}\label{eq:c.3}
\left(\int_{|\xi|\le 1}\int\bra{\tau}^{2\alpha} \left| \int \frac{\wt{w}_2(\tau',\xi)}{i(\tau'-\xi^5)}\wh{\psi}(\tau - \tau') \; d\tau'\right|^2 \; d\tau d\xi\right)^{\frac12} \lesssim \norm{w}_{X^{s,-b}}
\end{equation}
and
\begin{equation}\label{eq:c.4}
\Big(\int_{|\xi| > 1}|\xi|^{2s}\int\bra{\tau-\xi^5}^{2b} \Big| \int \frac{\wt{w}_2(\tau',\xi)}{i(\tau'-\xi^5)}\wh{\psi}(\tau - \tau') \; d\tau'\Big|^2 \; d\tau d\xi\Big)^{\frac12}\lesssim \norm{w}_{X^{s,-b}}.
\end{equation}
It follows the similar way used in the proof of (b). In fact, the proofs of \eqref{eq:c.3} and \eqref{eq:c.4}  are much simpler and easier than the proof of (b), since $L^2$ integral with respect to $\xi$ is negligible and hence it is enough to consider the relation between $\tau - \xi^5$ and $\tau'-\xi^5$. Thus, we omit the details and we have
\[\norm{\psi \mathcal{D}w}_{X^{s,b} \cap D^{\alpha}} \lesssim \norm{w}_{X^{s,-b}}.\]
\end{proof}

\begin{lemma}\label{edbf}$\;$\\
	\begin{itemize}
\item[(a)] \emph{(Space traces)}  Let $-\frac92 < s < 5$\footnote{The restriction of regularity makes the range of $\lambda$ non-empty.}. For $\max(s-\frac{9}{2}, -4) <\lambda< \min(s+\frac{1}{2}, \frac12)$, we have
		$$\|\psi(t)\mathcal{L}_{\pm}^{\lambda}f(t,x)\|_{C\big(\mathbb{R}_t;\,H^s(\mathbb{R}_x)\big)}\leq c \|f\|_{H_0^\frac{s+2}{5}(\mathbb{R}^+)};$$
\item[(b)] \emph{((Derivatives) Time traces)} For $-4+j<\lambda<1+j$, $j=0,1,2$, we have 
		\begin{equation}\label{eq:(b)0}
		\|\psi(t)\partial_x^j\mathcal{L}_{\pm}^{\lambda}f(t,x)\|_{C\big(\mathbb{R}_x;\,H_0^{\frac{s+2-j}{5}}(\mathbb{R}_t^+)\big)}\leq c \|f\|_{H_0^\frac{s+2}{5}(\mathbb{R}^+)};
		\end{equation}
\item[(c)] \emph{(Bourgain spaces)} Let $-7 < s < \frac52$\footnote{The restriction of regularity makes the range of $\lambda$ non-empty.} and $b < \frac12 < \alpha < 1-b$. For $\max(s-2, -\frac{13}{2}) <\lambda< \min(s+\frac{1}{2}, \frac12)$, we have
		\[\|\psi(t)\mathcal{L}_{\pm}^{\lambda}f(t,x)\|_{X^{s,b}\cap D^{\alpha}}\leq c \|f\|_{H_0^\frac{s+2}{5}(\mathbb{R}^+)}.\]
\end{itemize}

\end{lemma}

\begin{proof}
The proof is based on the argument in \cite{Holmerkdv} (see also \cite{Cavalcante}). We only consider $\mathcal{L}_-^{\lambda}$ for notational simplicity.

\textbf{(a).} By density, we may assume that $f\in C_{0,c}^{\infty}(\mathbb{R}^{+})$. Moreover, from the definition of $\mathcal{L}_-^{\lambda}$, it suffices to consider $\mathcal{L}_{-}^{\lambda}f(t,x)$ (removing $\psi$) for $\supp f \subset [0,1]$, thanks to Lemma \ref{cut}.

From \eqref{transformada} \eqref{eq:BFOC} and \eqref{eq:BFO}, we see that 
\[\mathcal{F}_x(\mathcal{L}_-^{\lambda}f)(t,\xi)=Me^{-\frac{i\pi\lambda}{2}}(\xi-i0)^{-\lambda}\int_0^te^{i(t-t')\xi^5}\mathcal{I}_{-\frac{\lambda}{5}-\frac{4}{5}}f(t') \;dt'.\] 
For fixed $t$, the change of variables ($\eta=\xi^5$), \eqref{transformada2} and the definition of the Fourier transform give
	\begin{eqnarray*}
		\|\mathcal{L}_-^{\lambda}f(t,\cdot)\|_{H^s(\mathbb{R})}^2&\leq& c \int_{\eta}|\eta|^{-\frac{2\lambda}{5}-\frac{4}{5}}\langle\eta\rangle^{\frac{2s}{5}}\left|\int_0^te^{i(t-t')\eta}\mathcal{I}_{-\frac{\lambda}{5}-\frac{4}{5}}f(t')dt'\right|^2d\eta\\
		&=&c\int_{\eta}|\eta|^{-\frac{2\lambda}{5}-\frac{4}{5}}\langle\eta\rangle^{\frac{2s}{5}}\left|\big(\chi_{(-\infty,t)}\mathcal{I}_{-\frac{\lambda}{5}-\frac{4}{5}}f\big)^{\widehat{}}(\eta)\right|^2d\eta.
	\end{eqnarray*}
The condition $-1 < -\frac{2\lambda}{5}-\frac{4}{5}$ ($\Leftrightarrow \lambda < \frac12$) enables us to replace $|\eta|^{-\frac{2\lambda}{5}-\frac{4}{5}}$ by $\bra{\eta}^{-\frac{2\lambda}{5}-\frac{4}{5}}$ thanks to Lemma \ref{sobolev0} (for $-1 < -\frac{2\lambda}{5}-\frac{4}{5} \le 0$) and the fact $|\eta|^{-\frac{2\lambda}{5}-\frac{4}{5}} \le \bra{\eta}^{-\frac{2\lambda}{5}-\frac{4}{5}}$ (for $0 < -\frac{2\lambda}{5}-\frac{4}{5}$). Moreover, Lemmas \ref{sobolevh0} (under the condition $-1 < -\frac{2\lambda}{5}-\frac{4}{5} + \frac{2s}{5} < 1$ for removing $\chi_{(-\infty,t)})$ and \ref{lio} (under the condition $-4 < \lambda$) yield
\[
\begin{aligned}
\int_{\eta}|\eta|^{-\frac{2\lambda}{5}-\frac{4}{5}}\langle\eta\rangle^{\frac{2s}{5}}\left|(\chi_{(-\infty,t)}\mathcal{I}_{-\frac{\lambda}{5}-\frac{4}{5}}f)^{\widehat{}}(\eta)\right|^2d\eta &\leq c\int_{\eta}\langle\eta\rangle^{\frac{2s}{5}-\frac{2\lambda}{5}-\frac{4}{5}}\left|(\chi_{(-\infty,t)}\mathcal{I}_{-\frac{\lambda}{5}-\frac{4}{5}}f)^{\widehat{}}(\eta)\right|^2d\eta\\
&\leq c\norm{\mathcal{I}_{-\frac{\lambda}{5}-\frac{4}{5}}f}_{H^{\frac{2s}{5}-\frac{2\lambda}{5}-\frac{4}{5}}}^2 \leq c \|f\|_{H_0^{\frac{s+2}{5}}}^2,	
\end{aligned}
\]
which proves (a) thanks to the definition of $H_0^s(\R^+)$-norm. 

\textbf{(b).} A direct calculation gives
\[\px^j\mathcal{L}_{\pm}^{\lambda}f = \mathcal{L}_{\pm}^{\lambda - j}(\mathcal{I}_{-\frac{j}{5}}f).\]
From this and Lemma \ref{lio}, it suffices to show \eqref{eq:(b)0} only for $j=0$. 

Lemma \ref{cut} ensures us to ignore the cut-off function $\psi$. The change of variables $t\rightarrow  t-t'$ gives
\begin{eqnarray*}
	& &(I-\partial_t^2)^{\frac{s+2}{10}}\left(\frac{x_{-}^{\lambda-1}}{\Gamma(\lambda)}*\int_{-\infty}^te^{-i(t-t')\partial_x^5}\delta(x)h(t')dt'\right)\\      
	& &\quad=\left(\frac{x_{-}^{\lambda-1}}{\Gamma(\lambda)}*\int_{-\infty}^te^{-i(t-t')\partial_x^5}\delta(x)(I-\partial_{t'}^2)^{\frac{s+2}{10}}h(t')dt'\right).
\end{eqnarray*}
Thus, it suffices to prove
\begin{equation}\label{eq:(b)}
\left\|\int_{\xi}e^{ix\xi}(\xi-i0)^{-\lambda}\int_{-\infty}^te^{i(t-t')\xi^5}(\mathcal{I}_{-\frac{\lambda}{5}-\frac{4}{5}}f)(t')dt'd\xi\right\|_{L_x^{\infty}L_t^2(\mathbb{R})} \leq c \|f\|_{L_t^2(\mathbb{R}^+)},
\end{equation}
thanks to $\pt^{\sigma}(\mathcal{I}_{\alpha}f) = \mathcal{I}_{\alpha}(\pt^{\sigma}f)$. We use $\chi_{(-\infty,t)}=\frac{1}{2}\mbox{sgn}(t-t')+\frac{1}{2}$ to obtain
\[\begin{aligned}
\int_{\xi}e^{ix\xi}(\xi-i0)^{-\lambda}&\int_{-\infty}^te^{i(t-t')\xi^5}(\mathcal{I}_{-\frac{\lambda}{5}-\frac{4}{5}}f)(t') \; dt'd\xi\\
=&~{}\frac12\int_{\xi}e^{ix\xi}(\xi-i0)^{-\lambda}\int_{-\infty}^{\infty}\mbox{sgn}(t-t')e^{i(t-t')\xi^5}(\mathcal{I}_{-\frac{\lambda}{5}-\frac{4}{5}}f)(t') \; dt'd\xi\\
&+\frac12\int_{\xi}e^{ix\xi}(\xi-i0)^{-\lambda}\int_{-\infty}^{\infty}e^{i(t-t')\xi^5}(\mathcal{I}_{-\frac{\lambda}{5}-\frac{4}{5}}f)(t') \; dt'd\xi\\
:=& I(t,x)+II(t,x).
\end{aligned}\]

We first deal with $I(t,x)$. We can rewrite $I$ as follows: 
\[I(t,x)=\frac12\int_{\xi}e^{ix\xi}(\xi-i0)^{-\lambda} \left((e^{i\cdot\xi^5}\mbox{sgn}(\cdot))*\mathcal{I}_{-\frac{\lambda}{5}-\frac{4}{5}}f\right)(t) \; d\xi.\]
A direct calculation gives
$$\mathcal{F}_t\left((e^{i\cdot\xi^5}\mbox{sgn}(\cdot))*\mathcal{I}_{-\frac{\lambda}{5}-\frac{4}{5}}f\right)(\tau)=\frac{(\tau-i0)^{\frac{4+\lambda}{5}}\hat{f}(\tau)}{i(\tau-\xi^5)},$$ 
which, in addition to Fubini theorem and DCT, implies
\begin{equation*}
I(t,x)=\int_{\tau}e^{it\tau}\lim_{\epsilon\rightarrow 0}\int_{|\tau-\xi^5|>\epsilon}\frac{e^{ix\xi}(\tau-i0)^{\frac{\lambda+4}{5}}(\xi-i0)^{-\lambda}}{i(\tau-\xi^5)}\wh{f}(\tau) \; d\xi d\tau.
\end{equation*}
Thus, once we show that the function 
$$g(\tau):=\lim_{\epsilon\rightarrow 0}\int_{|\tau-\xi^5|>\epsilon}\frac{e^{ix\xi}(\tau-i0)^{\frac{\lambda+4}{5}}(\xi-i0)^{-\lambda}}{(\tau-\xi^5)} \; d\xi$$
is bounded independently of $\tau$ variable, the Plancherel's theorem enables us to obtain \eqref{eq:(b)}. The change of variables $\xi \mapsto |\tau|^{\frac{1}{5}}\xi$  and the fact from \eqref{transformada2} that
\[(|\tau|^{\frac{1}{5}}\xi-i0)^{-\lambda}=|\tau|^{-\frac{\lambda}{5}}(\xi_{+}^{-\lambda}+e^{i\pi\lambda}\xi_{-}^{-\lambda})\]
gives
\begin{eqnarray*}
	g(\tau) &=&\chi_{\{\tau>0\}}\int_{\xi}e^{ix|\tau|^{\frac{1}{5}}\xi} \; \frac{\xi_{+}^{-\lambda}+e^{i\pi\lambda}\xi_{-}^{-\lambda}}{1-\xi^5} \; d\xi - e^{-\frac{i\pi(\lambda + 4)}{5}}\chi_{\{\tau<0\}}\int_{\xi}e^{ix|\tau|^{\frac{1}{5}}\xi} \; \frac{\xi_{+}^{-\lambda}+e^{i\pi\lambda}\xi_{-}^{-\lambda}}{1+\xi^5} \; d\xi\\
	&:=&g_1 - e^{-\frac{i\pi(\lambda + 4)}{5}}g_2.
\end{eqnarray*}
We only consider $g_1$, since $g_2$ is uniformly bounded in $\tau$ for $-4 < \lambda < 1$. Let $\zeta\in C^{\infty}(\mathbb{R})$ such that $\zeta(\xi)=1$ in $[\frac{3}{4},\frac{4}{3}]$ and $\zeta(t)=0$ outside $(\frac{1}{2},\frac{3}{2})$. Then we obtain
\begin{eqnarray*}
	g_1&=&\chi_{\{\tau>0\}}\int_{\xi}e^{ix|\tau|^{\frac{1}{5}}\xi} \zeta(\xi) \frac{\xi_{+}^{-\lambda}}{1-\xi^5} \; d\xi+ \chi_{\{\tau>0\}}\int_{\xi}e^{ix|\tau|^{\frac{1}{5}}\xi} (1-\zeta(\xi)) \frac{\xi_{+}^{-\lambda}+e^{i\pi\lambda}\xi_{-}^{-\lambda}}{1-\xi^5} \; d\xi\\
	&=&g_{11}+g_{12}.
\end{eqnarray*}
It is clear that $g_{12}$ is bounded independently of $\tau$ when $\lambda>-4$, and hence it remains to deal with $g_{11}$. Let
\[\wh{\Theta}(\xi) = \frac{\zeta(\xi) \xi_+^{-\lambda}}{1+\xi + \xi^2 + \xi^3 + \xi^4} \quad \mbox{and} \quad \wh{\Psi}(\xi) = \frac{1}{i(\xi - 1)}.\]
We remark that $\wh{\Theta}$ is a Schwartz function, and hence $\Theta \in \Sch(\R)$. Moreover, we immediately know from the fact $\ft_x[\mbox{sgn}(x)](\xi) = \frac{2}{i\xi}$ that
\[\Psi(x) = \frac12 e^{ix} \mbox{sgn}(x).\]
Then, $g_{11}$ can be written as
\[g_{11}(\tau) = -i\chi_{\{\tau>0\}}\int_{\xi}e^{ix|\tau|^{\frac{1}{5}}\xi}\wh{\Theta}(\xi)\wh{\Psi}(\xi) \; d\xi = -2i\pi \chi_{\{\tau>0\}}(\Theta \ast \Psi)(|\tau|^{\frac15}x),\]
which implies
\[
\begin{aligned}
|g_{11}(\tau)| &\lesssim \left|\int \Theta(y)\Psi(|\tau|^{\frac15}x - y) \; dy  \right|\\
&\lesssim \int |\Theta(y)| \; dy \lesssim_{\zeta} 1.
\end{aligned}
\]

We now deal with $II(t,x)$. The definition of Fourier transform, \eqref{transformada2}, the changes of variables ($\eta=\xi^5$) and contour yield
\[\begin{aligned}
	II(t,x)=&~{}\frac12\int_{\xi}e^{ix\xi}e^{it\xi^5}(\xi^5-i0)^{\frac{\lambda+4}{5}}\wh{f}(\xi^5)(\xi-i0)^{-\lambda} \; d\xi\\
	=&~{}\frac12\int_{\eta}e^{it\eta}e^{ix\eta^{\frac{1}{5}}}(\eta-i0)^{\frac{\lambda+4}{5}}(\eta^{\frac{1}{5}}-i0)^{-\lambda}\eta^{-\frac{4}{5}}\wh{f}(\eta) \;d\eta\\
	=&~{}cf(t),
\end{aligned}\]
for some $c \in \mathbb{C}$, which implies $\|II(\cdot,x)\|_{L_t^2} \lesssim  \|f\|_{L_t^2}$. Therefore, we complete the proof.	
	
\textbf{(c).} A direct calculation gives
\begin{equation*}
\mathcal{F}_x(\psi(t)\mathcal{L}_-^{\lambda}f)(t,\xi)=Me^{-\frac{i\pi\lambda}{2}}e^{\frac{i\pi(\lambda+4)}{10}}(\xi-i0)^{-\lambda}\psi(t)e^{it\xi^5}\int\frac{e^{it(\tau'-\xi^5)}-1}{i(\tau'-\xi^5)}(\tau'-i0)^{\frac{\lambda}{5}+\frac{4}{5}}\wh{f}(\tau')\; d\tau',
\end{equation*}
which can be divided into the followings: 
\begin{equation*}
\wh{f}_1(t,\xi)=Me^{-\frac{i\pi\lambda}{2}}e^{\frac{i\pi(\lambda+4)}{10}}(\xi-i0)^{-\lambda}\psi(t)\int\frac{e^{it\tau'}-e^{it\xi^5}}{i(\tau'-\xi^5)}\theta(\tau'-\xi^5)(\tau'-i0)^{\frac{\lambda}{5}+\frac{4}{5}}\wh{f}(\tau')\; d\tau',
\end{equation*} 
\begin{equation*}
\wh{f}_2(t,\xi)=Me^{-\frac{i\pi\lambda}{2}}e^{\frac{i\pi(\lambda+4)}{10}}(\xi-i0)^{-\lambda}\psi(t)\int\frac{e^{it\tau'}}{i(\tau'-\xi^5)}(1-\theta(\tau'-\xi^5))(\tau'-i0)^{\frac{\lambda}{5}+\frac{4}{5}}\wh{f}(\tau')\; d\tau'
\end{equation*} 
and
\begin{equation*}
\wh{f}_3(t,\xi)=Me^{-\frac{i\pi\lambda}{2}}e^{\frac{i\pi(\lambda+4)}{10}}(\xi-i0)^{-\lambda}\psi(t)\int\frac{e^{it\xi^5}}{i(\tau'-\xi^5)}(1-\theta(\tau'-\xi^5))(\tau'-i0)^{\frac{\lambda}{5}+\frac{4}{5}}\wh{f}(\tau')\; d\tau',
\end{equation*} 
where $\theta \in \mathcal{S}(\R)$ such that $\theta(\tau)=1$ for $|\tau|\leq 1$ and $\theta(\tau)=0$ for $|\tau|\geq 2$. We know that $\psi(t)\mathcal{L}_-^{\lambda}f= f_1 + f_2 - f_3$. 

For $f_1$, we use the same argument for $w_1$ in the proof of Lemma \ref{duhamel} (c)\footnote{Here, it is not necessary to distinguish $X^{s,b}$ and $D^{\alpha}$ portions. The same will be true of $f_3$}. By the Taylor series expansion for $e^{it(\tau'-\xi^5)}$ at $it(\tau'-\xi^5) = 0$, we write 
\[
\psi(t)\mathcal{L}_-^{\lambda}f_1(t,x) = c\sum_{k=1}^{\infty}\frac{i^{k-1}}{k!}\psi^k(t)e^{t\partial_x^5}F_1^k(x),
\]
for some constant $c \in \mathbb{C}$, where $\psi^k(t) = t^k\psi(t)$ and 
\[
\wh{F}_1^k(\xi) = (\xi-i0)^{-\lambda}\int \theta(\tau'-\xi^5)(\tau'-\xi^5)^{k-1}\tau'^{\frac{\lambda}{5}+\frac{4}{5}}\wh{f}(\tau') \; d\tau'.
\]	
By \eqref{transformada2}, Lemma \eqref{grupo} (c) and the definition of $\theta$, it is enough to show that
\begin{equation}\label{crb100}
\int_{\xi}\langle \xi\rangle^{2s}|\xi|^{-2\lambda}\left|\int_{|\tau'-\xi^5|\leq 1}|\tau'-\xi^5|^{k-1}|\tau'|^{\frac{\lambda+4}{5}}|\wh{f}(\tau')| \; d\tau' \right|^2 d\xi \lesssim \norm{f}_{H^{\frac{s+2}{5}}}^2.
\end{equation}
Since both $|\xi|^{-2\lambda}$ and $|\tau'|^{\frac{2(\lambda+4)}{5}}$ for $-\frac{13}{2} < \lambda < \frac12$ are integrable on the regions $|\xi| \le 1$ and $|\tau'| \lesssim 1$ ($|\xi| \le 1$ and $|\tau'-\xi^5| \le 1$ imply $|\tau'| \lesssim 1$), respectively, we obtain \eqref{crb100} by taking Cauchy-Schwarz inequality in $\tau'$. Thus, we assume $|\xi| > 1$, which in addition to $|\tau'-\xi^5| \le 1$ implies $|\tau'| \sim |\xi|^5 > 1$. Let $\wh{f}^{\ast}(\tau') = \bra{\tau'}^{\frac{s+2}{5}}\wh{f}(\tau')$. Then, the change of variables ($\xi^5 \mapsto \eta$) gives
\[
\begin{aligned}
\mbox{LHS of } \eqref{crb100} &\lesssim \int_{|\xi| > 1}\xi^4 |M\wh{f}^{\ast}(\xi^5)|^2 \; d\xi\\
&\lesssim \int_{|\eta| > 1}|M\wh{f}^{\ast}(\eta)|^2 \; d\eta\\
&\lesssim \norm{f^{\ast}}_{L^2} = \norm{f}_{H^{\frac{s+2}{5}}},
\end{aligned}
\]
where $M\wh{f}^{\ast}$ is the Hardy-Littlewood maximal function of $\wh{f}^{\ast}$. Hence we have 
\[\norm{f_1}_{X^{s,b} \cap D^{\alpha}} \lesssim \norm{f}_{H^{\frac{s+2}{5}}}.\]

For $f_2$, from \eqref{transformada2}, the definition of inverse Fourier transform and Lemma \ref{lem:Xsb}, we have
\begin{eqnarray*}
	\|f_2\|_{X^{s,b}}^2&\lesssim& \int\int\langle \xi\rangle^{2s}|\xi|^{-2\lambda}\langle \tau-\xi^5\rangle^{2b}\frac{(1-\theta(\tau-\xi^5))^2}{|\tau-\xi^5|^2}|\tau|^{\frac{2\lambda+8}{5}}|\wh{f}(\tau)|^2 \; d\tau d\xi\\
	&\lesssim&\int  |\tau|^{\frac{2\lambda+8}{5}}\left( \int\frac{\langle \xi\rangle^{2s}|\xi|^{-2\lambda} }{\langle \tau-\xi^5\rangle^{2-2b}}\; d\xi\right)|\hat{f}(\tau)|^2 \; d\tau.
\end{eqnarray*}
Thus, by the change of variables ($\eta = \xi^5$) and Lemma \ref{sobolev0} for $-\frac{13}{2} < \lambda$ (we may assume $\supp f \subset [0,1]$, thanks to Lemma \ref{cut}), it suffices to show
\[I(\tau)=\int\frac{|\eta|^{-\frac{4}{5}-\frac{2\lambda}{5}}\langle \eta\rangle^{\frac{2s}{5}}}{\langle \tau-\eta\rangle^{2-2b}} \; d\eta \lesssim  \langle\tau\rangle^{\frac{2s}{5}-\frac{2\lambda}{5}-\frac{4}{5}}.\]
When $|\tau|\leq2$, we have $\langle\tau-\eta\rangle\sim\langle\eta\rangle$. For $s-2 < \lambda<\frac{1}{2}$ and $b < \frac12$, we have
\begin{equation*}
I(\tau)\lesssim \int_{|\eta|\leq 1} |\eta|^{\frac{-4-2\lambda}{5}}+\int\frac{d\eta}{\langle \eta\rangle^{2-2b-\frac{2s}{5}+\frac{4}{5}+\frac{2\lambda}{5}}}\lesssim 1.
\end{equation*}
When $|\tau|>2$, we divide the integral region in $\eta$ into $|\eta|< \frac{|\tau|}{2}$ and  $|\eta| \ge \frac12|\tau|$. In the former case, for $b<\frac12$ and $\lambda < \min(\frac12, s+\frac12)$, we have
\begin{equation*}
I(\tau)\lesssim \langle\tau\rangle^{2b-2}\left(\int_{|\eta| \le 1}|\eta|^{-\frac{4}{5}-\frac{2\lambda}{5}} \; d\eta +\int_{1 < |\eta|\leq \frac{|\tau|}{2}}|\eta|^{\frac{2s-4-2\lambda}{5}}d\eta \right) \lesssim \langle\tau\rangle^{\frac{2s}{5}-\frac{4}{5}-\frac{2\lambda}{5}}.
\end{equation*} 
On the other hand, we have in the second case that $|\tau-\eta| \geq \frac{1}{2} |\tau| > 1$. Then, for $s-2 < \lambda$ and $b<\frac{1}{2}$, we have 
\begin{equation*}
\begin{aligned}
I(\tau) &\lesssim \langle\tau\rangle^{\frac{2s}{5}-\frac{4}{5}-\frac{2\lambda}{5}}\int \frac{d\eta}{\langle\tau-\eta\rangle^{2-2b}} \\
&\lesssim \langle\tau\rangle^{\frac{2s}{5}-\frac{4}{5}-\frac{2\lambda}{5}}\int_{|s| > 1} \frac{ds}{|s|^{2-2b}}\\
&\lesssim \langle\tau\rangle^{\frac{2s}{5}-\frac{4}{5}-\frac{2\lambda}{5}}.
\end{aligned}
\end{equation*}
Note for the $D^{\alpha}$-portion that it suffices to show 
\[I(\tau)\lesssim \langle\tau\rangle^{2\alpha-2} \int_{|\eta| \le 1}|\eta|^{-\frac{4}{5}-\frac{2\lambda}{5}} \; d\eta \lesssim \langle\tau\rangle^{\frac{2s}{5}-\frac{4}{5}-\frac{2\lambda}{5}}.\] 
It immediately follows the same way as above for $\lambda < \min(\frac12, s+\frac12)$ thanks to $b < \frac12 < \alpha < 1-b$. Hence we have
\[\norm{f_2}_{X^{s,b} \cap D^{\alpha}} \lesssim \norm{f}_{H^{\frac{s+2}{5}}}.\]

For $f_3$, similarly as for $f_1$, it suffices to show 
\begin{equation}\label{23110}
\int\langle \xi\rangle^{2s}|\xi|^{-2\lambda}\left|\int (1-\theta(\tau'-\xi^5))|\tau'-\xi^5|^{-1}|\tau'|^{\frac{\lambda+4}{5}}|\wh{f}(\tau')| \; d\tau' \right|^2 d\xi \lesssim \norm{f}_{H^{\frac{s+2}{5}}}^2.
\end{equation}
When $|\xi| \le 1$, since $|\xi|^{-2\lambda}$ is integrable for $\lambda < \frac12$, we may ignore the integration in $\xi$. Moreover, if $|\tau'| \le 1$, $|\tau'|^{\frac{2(\lambda+4)}{5}}$ for $-\frac{13}{2} < \lambda$ is integrable, and hence we get \eqref{23110}. On the region $|\tau'| > 1$, since $|\tau' -\xi^5| \sim |\tau'|$ and $|\tau'|^{\frac{-s+\lambda -3}{5}}$ is $L^2$ integrable for $\lambda < s+\frac12$, we also get \eqref{23110} by using the Cauchy-Schwarz inequality in $\tau'$. On the other hand, when $|\xi| > 1$ and $|\tau'| \le 1$, since $|\tau' -\xi^5| \sim |\xi|^5$ and $|\xi|^{2s-2\lambda-10}$ is integrable for $s - \frac92 < s-2 < \lambda$, we also get \eqref{23110}. We thus consider the region $|\xi| > 1$ and $|\tau'| > 1$. 

There are two possibility: \textbf{I.} $|\tau'| \le \frac12|\xi|^5$, \textbf{II.} $\frac12|\xi|^5 <|\tau'|$. In view of the proof of Lemma 5.8 (d) in \cite{Holmerkdv}, one can replace $\frac{1-\theta(\tau'-\xi^5)}{\tau'-\xi^5}$ by $\beta(\tau'-\xi^5)$ for some $\beta \in \Sch(\R)$. Hence, the left-hand side of \eqref{23110} is dominated by
\begin{equation}\label{eq:23110}
\int_{|\xi| > 1}|\xi|^{2s-2\lambda}\left|\int_{|\tau'|>1} |\tau'-\xi^5|^{-N}|\tau'|^{\frac{\lambda+4}{5}}|\wh{f}(\tau')| \; d\tau' \right|^2 d\xi,
\end{equation}
for $N \ge 0$. By taking the Cauchy-Schwarz inequality and choosing $N = N(s,\lambda) \gg 1$, we have \eqref{23110} for both cases. Indeed, we have for the case $I$ (in this case, we have $|\tau'-\xi^5| \sim |\xi|^5$) that 
\[\eqref{eq:23110} \lesssim \int_{|\xi| > 1}|\xi|^{2s-2\lambda-10N}\left|\int_{1 < |\tau'| \le \frac12|\xi|^5} |\tau'|^{\frac{-s + \lambda+2}{5}}|\tau'|^{\frac{s+2}{5}}|\wh{f}(\tau')| \; d\tau' \right|^2 d\xi \lesssim \norm{f}_{H^{\frac{s+2}{5}}}^2\]
and for the case $II$ (in this case, we have $|\tau'-\xi^5| \sim |\tau'|$) that
\[\eqref{eq:23110} \lesssim \int_{|\xi| > 1}|\xi|^{2s-2\lambda}\left|\int_{\frac12|\xi|^5 <  |\tau'|} |\tau'|^{\frac{-s + \lambda+2 - 5N}{5}}|\tau'|^{\frac{s+2}{5}}|\wh{f}(\tau')| \; d\tau' \right|^2 d\xi \lesssim \norm{f}_{H^{\frac{s+2}{5}}}^2.\]
Hence we have
\[\norm{f_2}_{X^{s,b} \cap D^{\alpha}} \lesssim \norm{f}_{H^{\frac{s+2}{5}}},\]
and we complete the proof of (c).
\end{proof}

\section{Bilinear estimates}\label{sec:bilinear}
In this section, we are going to prove bilinear estimates.
\subsection{$L^2$-block estimates}
For $\xi_1,\xi_2 \in \R$, let
\begin{equation}\label{eq:resonant function}
\begin{aligned}
H = H(\xi_1,\xi_2) &= (\xi_1 + \xi_2)^5 - \xi_1^5 -\xi_2^5 \\
&=\frac52\xi_1\xi_2(\xi_1+\xi_2)(\xi_1^2+\xi_2^2+(\xi_1+\xi_2)^2)
\end{aligned}
\end{equation}
be the resonance function, which plays an crucial role in the bilinear $X^{s,b}$-type estimates. For compactly supported functions $f,g,h \in L^2(\R^2)$, we define
\[J(f,g,h) = \int_{\R^4} f(\zeta_1,\xi_1)g(\zeta_2,\xi_2)h(\zeta_1+\zeta_2+H(\xi_1,\xi_2), \xi_1+\xi_2) \; d\xi_1d\xi_2d\zeta_1\zeta_2. \]
By the change of variables in the integration, we know
\[J(f,g,h)=J(\wt{g},h,f)=J(h,\wt{f},g),\]
where $\tilde{f}(\zeta,\xi)=f(-\zeta,-\xi)$. Note from the definition of convolution operation that the first two components in the functional $J$ can be freely changed each other without any difference.

From the identities
\begin{equation}\label{eq:symmetry1}
\xi_1 + \xi_2 = \xi_3
\end{equation}
and
\begin{equation}\label{eq:symmetry2}
(\tau_1 - \xi_1^5)+(\tau_2 - \xi_2^5) = (\tau_3 - \xi_3^5) + H(\xi_1,\xi_2)
\end{equation}
on the support of $J(f^{\sharp},g^{\sharp},h^{\sharp})$, where $f^{\sharp}(\tau,\xi) = f(\tau-\xi^5,\xi)$ with the property $\norm{f}_{L^2} = \norm{f^{\sharp}}_{L^2}$, we see that $J(f^{\sharp},g^{\sharp},h^{\sharp})$ vanishes unless
\begin{equation}\label{eq:support property}
\begin{array}{c}
2^{k_{max}} \sim 2^{k_{med}} \gtrsim 1\\
2^{j_{max}} \sim \max(2^{j_{med}}, |H|).
\end{array}
\end{equation}

\begin{lemma}\label{lem:block estimate} 
	Let $k_i \in \Z,j_i\in \Z_+,i=1,2,3$. Let $f_{k_i,j_i} \in L^2(\R\times\R) $ be nonnegative functions supported in $[2^{k_i-1},2^{k_i+1}]\times I_{j_i}$.
	
	(a) For any $k_1,k_2,k_3 \in \Z$ with $|k_{max}-k_{min}| \le 5$ and $j_1,j_2,j_3 \in \Z_+$, then we have
	\[J(f_{k_1,j_1},f_{k_2,j_2},f_{k_3,j_3}) \lesssim 2^{j_{min}/2}2^{j_{med}/4}2^{- \frac34 k_{max}}\prod_{i=1}^3 \|f_{k_i,j_i}\|_{L^2}.\]
	
	(b) If $2^{k_{min}} \ll 2^{k_{med}} \sim 2^{k_{max}}$, then for all $i=1,2,3$ we have
	\[J(f_{k_1,j_1},f_{k_2,j_2},f_{k_3,j_3}) \lesssim 2^{(j_1+j_2+j_3)/2}2^{-3k_{max}/2}2^{-(k_i+j_i)/2}\prod_{i=1}^3 \|f_{k_i,j_i}\|_{L^2}.\]
	
	(c) For any $k_1,k_2,k_3 \in \Z$ and $j_1,j_2,j_3 \in \Z_+$, then we have
	\[J(f_{k_1,j_1},f_{k_2,j_2},f_{k_3,j_3}) \lesssim 2^{j_{min}/2}2^{k_{min}/2}\prod_{i=1}^3 \|f_{k_i,j_i}\|_{L^2}.\]
\end{lemma}

$L^2$-block estimates of fifth-order equations on $\R$ and $\T$ have been used in the proof of multilinear estimates, see \cite{CLMW2009, CG2011, GKK2013, KP2015, Kwak2016, Kwak2015}. The proof was first shown by Chen, Li, Miao and Wu \cite{CLMW2009}. But there was an error in the \emph{high $\times$ high $\Rightarrow$ high} case and was corrected in \cite{CG2011}. See \cite{CLMW2009, CG2011} for the proof. 

\subsection{Control of $X^{s,-b}$ norm of nonlinear terms}

\begin{proposition}\label{prop:bi1}
	For $-7/4 < s $, there exists $b = b(s) < 1/2$ such that for all $\alpha > 1/2$, we have
	\begin{equation}\label{eq:bilinear1}
	\norm{\px(uv)}_{X^{s,-b}} \le c\norm{u}_{X^{s,b} \cap D^{\alpha}}\norm{v}_{X^{s,b} \cap D^{\alpha}}.
	\end{equation}
\end{proposition}

\begin{proof}
	Let
	\begin{equation}\label{eq:restriction1}
	\wt{f}_1(\tau_1,\xi_1) = \beta_1(\tau_1,\xi_1)\wt{u}(\tau_1,\xi_1) \hspace{1em} \mbox{and} \hspace{1em} \wt{f}_2(\tau_2,\xi_2) = \beta_2(\tau_2,\xi_2)\wt{v}(\tau_2,\xi_2),
	\end{equation}
	where
	\begin{equation}\label{eq:restriction2}
	\beta_i(\tau_i,\xi_i) = \bra{\tau_i-\xi_i^5}^b + \mathbf{1}_{|\xi_i|\le1}(\xi_i)\bra{\tau_i}^{\alpha}, \hspace{1em} i=1,2.
	\end{equation}
	Note that
	\begin{equation}\label{eq:weight}
	\begin{aligned}
	\frac{1}{\beta_i(\tau_i,\xi_i)} \lesssim \begin{cases}\begin{array}{ll}\bra{\tau_i-\xi_i^5}^{-b}, & \hspace{1em} \mbox{when} \hspace{1em} |\xi_i| > 1, \\ \bra{\tau_i}^{-\alpha}, & \hspace{1em} \mbox{when} \hspace{1em} |\xi_i| \le 1.  \end{array}\end{cases}
	\end{aligned}
	\end{equation}
	From the duality in addition to the fact that $f_1, f_2 \in L^2 \Leftrightarrow u, v \in X^{s,b} \cap D^{\alpha}$, \eqref{eq:bilinear1} is equivalent to 
	\begin{equation}\label{eq:bilinear1-1}
	\iint\limits_{\substack{\xi_1+\xi_2=\xi \\ \tau_1+\tau_2=\tau}} \frac{|\xi|\bra{\xi}^s\wt{f}_1(\tau_1,\xi_1)\wt{f}_2(\tau_2,\xi_2)\wt{f}_3(\tau,\xi)}{\bra{\xi_1}^s\bra{\xi_2}^s\bra{\tau-\xi^5}^b\beta_1(\tau_1,\xi_1)\beta_2(\tau_2,\xi_2)} \lesssim \norm{f_1}_{L^2}\norm{f_2}_{L^2}\norm{f_3}_{L^2}.
	\end{equation}
	For $k_i,j_i \in \Z_+$, we make a dyadic decomposition of $f_i$, $i=1,2,3$, in both frequency and modulation modes into $f_{k_i,j_i}$, $i=1,2,3$, as $f_{k_i,j_i}(\tau,\xi) = \eta_{j_i}(\tau-\xi^5)\chi_{k_i}(\xi)\wt{f}_i(\tau,\xi)$. We prove \eqref{eq:bilinear1-1} by dividing the region of integration 
	\begin{equation}\label{eq:bilinear1-2}
	\iint\limits_{\substack{\xi_1+\xi_2=\xi \\ \tau_1+\tau_2=\tau}} \frac{|\xi|\bra{\xi}^s\wt{f}_1(\tau_1,\xi_1)\wt{f}_2(\tau_2,\xi_2)\wt{f}_3(\tau,\xi)}{\bra{\xi_1}^s\bra{\xi_2}^s\bra{\tau-\xi^5}^b\beta_1(\tau_1,\xi_1)\beta_2(\tau_2,\xi_2)}.
	\end{equation}
	into several regions associated to the relation of frequencies.
	
	\textbf{Case I.} \emph{high $\times$ high $\Rightarrow$ high} ($k_3 \ge 10, |k_3-k_1|,|k_3-k_2|\le 5$). In this case, we have $j_{max} \ge 5k_3 -5$ thanks to \eqref{eq:support property} and \eqref{eq:resonant function}. The change of variables in addition to \eqref{eq:weight} enables that \eqref{eq:bilinear1-2} is bounded by
	\[\sum_{\substack{k_3 \ge 10 \\|k_3-k_1| \le 5 \\ |k_3-k_2|\le 5}}\sum_{j_1,j_2,j_3 \ge 0} 2^{(1-s)k_3}2^{-b(j_1+j_2+j_3)}J(f_{k_1,j_1}^{\sharp},f_{k_2,j_2}^{\sharp},f_{k_3,j_3}^{\sharp}).\]
	From Lemma \ref{lem:block estimate} (a), the Cauchy-Schwarz inequality and \eqref{eq:dyadic X}, it suffices to show
	\begin{equation}\label{eq:bilinear1-3}
	\sum_{\substack{k_3 \ge 10 \\|k_3-k_1| \le 5 \\ |k_3-k_2|\le 5}}\sum_{j_1,j_2,j_3 \ge 0} 2^{2(1-s)k_3}2^{-2b(j_1+j_2+j_3)}2^{j_{min}}2^{j_{med}/2}2^{- \frac32 k_{max}} \lesssim 1.
	\end{equation}
	Without loss of generality, we may assume that $j_1 \le j_2 \le j_3$.  Given $-9/4 < s$, we can choose $\max(3/8 , 1/20-s/5) < b < 1/2$.  Performing the summation over $0 \le j_1 \le j_2 \le j_3$ in addition to $5k_3 - 5 \le j_3$ and $k_i$, $i=1,2,3$ yields
	\[\mbox{LHS of } \eqref{eq:bilinear1-3} \lesssim  \sum_{k_3 \ge 10} 2^{(1/2 -2s)k_3}2^{-10bk_3} \lesssim 1,\]
	which completes the proof of \eqref{eq:bilinear1-3}. 
	
	\textbf{Case II} \emph{high $\times$ low $\Rightarrow$ high} ($k_3 \ge 10, 0 \le k_1 \le k_3-5, |k_3-k_2| \le 5$).\footnote{Due to the symmetry, this case is exactly same as the case when $k_3 \ge 10, 0 \le k_2 \le k_3-5, |k_3-k_1| \le 5$.} We further divide the case into two cases: $k_1 = 0$ and $k_1 \ge 1$.
	
	\textbf{Case II-a} $k_1 = 0$. Without loss of generality, we may assume that $j_2 \le j_3$. Similarly as \textbf{Case I}, it suffices to show
	\begin{equation}\label{eq:bilinear1-4}
	\sum_{\substack{k_3 \ge 10 \\ |k_3-k_2|\le 5}}\sum_{j_1,j_2,j_3 \ge 0} 2^{2k_3}2^{-2\alpha j_1}2^{-2b(j_2+j_3)}2^{(j_1+j_2+j_3)}2^{-3k_{max}}2^{-(k_3+j_3)} \lesssim 1,
	\end{equation}
	thanks to \eqref{eq:weight} and Lemma \ref{lem:block estimate} (b).  Note that the estimate \eqref{eq:bilinear1-4} is irrelevant to the regularity $s$. By choosing $1/4 < b <1/2$,  we have
	\[\sum_{\substack{0 \le j_1 \\0 \le j_2 \le j_3}} 2^{(1-2\alpha)j_1}2^{(1-2b)j_2}2^{-2bj_3} \lesssim \sum_{0 \le j_1, j_3} 2^{(1-2\alpha)j_1}2^{(1-4b)j_3} \lesssim 1,\]
	for all $\alpha > 1/2$. Moreover, the estimate 
	\[\sum_{k_3 \ge 10 } 2^{-2k_3} \lesssim 1\]
	holds true. This completes the proof of \eqref{eq:bilinear1-4}.
	
	\textbf{Case II-b} $k_1 \ge 1$. In this case, we have from \eqref{eq:support property} and \eqref{eq:resonant function} that $j_{max} \ge 4k_3 + k_1 - 5$. Similarly as \textbf{Case I}, it suffices to show
	\begin{equation}\label{eq:bilinear1-5}
	\sum_{\substack{k_3 \ge 10 \\ 1 \le k_1 \le k_3 - 5 \\ |k_3-k_2|\le 5}}\sum_{j_1,j_2,j_3 \ge 0} 2^{2k_3}2^{-2sk_1}2^{-2b(j_1+j_2+j_3)}2^{(j_1+j_2+j_3)}2^{-3k_{max}}2^{-(k_i+j_i)} \lesssim 1,
	\end{equation}
	thanks to \eqref{eq:weight} and Lemma \ref{lem:block estimate} (b). Without loss of generality, we may assume that $j_2 \le j_3$. 
	
	If $j_1 \neq j_{max}$,  given $-7/2 < s$, we can choose $\max((4-s)/15 , 1/4) < b < 1/2$.  Performing the summation over $0 \le j_1, j_2 \le j_3$ in addition to $4k_3 + k_1 - 5 \le j_3$ after choosing $(k_i,j_i) = (k_3,j_3)$ yields
	\[\mbox{LHS of } \eqref{eq:bilinear1-4} \lesssim  \sum_{k_3 \ge 10  }\sum_{1 \le k_1 \le k_3 - 5} 2^{(6-24b)k_3}2^{(2-6b-2s)k_1} \lesssim 1.\]
	
	If $j_1 = j_{max}$,  given $-7/2 < s$, we can choose $\max((4-s)/15 , 7/24) < b < 1/2$.  Perform the summation over $0 \le j_2 \le j_3 \le j_1$ in addition to $4k_3 + k_1 - 5 \le j_1$ after choosing $(k_i,j_i) = (k_1,j_1)$ yields
	\[\mbox{LHS of } \eqref{eq:bilinear1-4} \lesssim  \sum_{k_3 \ge 10  }\sum_{1 \le k_1 \le k_3 - 5} 2^{(7-24b)k_3}2^{(1-6b-2s)k_1} \lesssim 1.\] 
Hence,  given $-7/2 < s$, by choosing $\max((4-s)/15 , 7/24) < b < 1/2$ , we can complete the proof of \eqref{eq:bilinear1-5}.  
	
	\textbf{Case III.} \emph{high $\times$ high $\Rightarrow$ low} ($k_2 \ge 10, |k_1-k_2| \le 5, 0 \le k_3 \le k_2 - 5$). We further divide the case into two cases: $k_3 = 0$ and $k_3 \ge 1$.
	
	\textbf{Case III-a} $k_3 = 0$. In this case, we decompose further the low frequency component $f_3=\sum_{l\leq 0}f_{3}^{l}$ with $f_{3}^l=\ft^{-1}1_{|\xi|\sim 2^l}\ft f_3$. Then, from \eqref{eq:weight}, \eqref{eq:bilinear1-2} is bounded by
	\begin{equation}\label{eq:hhl}
	\sum_{\substack{k_2 \ge 10 \\ |k_1-k_2| \le 5}}\sum_{l\le0}\sum_{j_1,j_2,j_3 \ge 0} 2^{-2sk_2}2^{l}2^{-b(j_1+j_2+j_3)}J(f_{k_1,j_1}^{\sharp},f_{k_2,j_2}^{\sharp},f_{l,j_3}^{\sharp}),
	\end{equation}
	where $f_{l,j_3}(\tau,\xi) = \eta_{j_3}(\tau-\xi^5)\wt{f}_{3}^l(\tau,\xi)$. Without loss of generality, we may assume that $j_1 \le j_2$.  
	
	If $j_3=j_{max}$, applying Lemma \ref{lem:block estimate} (b) to $J(f_{k_1,j_1}^{\sharp},f_{k_2,j_2}^{\sharp},f_{l,j_3}^{\sharp})$  and performing the Cauchy-Schwarz inequality in terms of $k_2,l,j_i's$ yield
	\[\eqref{eq:hhl} \lesssim \sum_{\substack{k_2 \ge 10 \\ |k_1-k_2| \le 5}}\sum_{l\le 0}\sum_{0 \le j_1 \le j_2 \le j_3} 2^{-4sk_2}2^{2l}2^{-2b(j_1+j_2+j_3)}2^{(j_1+j_2+j_3)}2^{-3k_{max}}2^{-(l+j_3)}.\]
	Given $s > -7/4$, we can choose $\max((5-4s)/24 , 1/3) < b < 1/2$.  Since $j_{max} \ge 4k_2 + l - 5$ and $1/3 <b < 1/2$, we have
	\[\sum_{\substack{0 \le j_1 \le j_2 \le j_3 \\4k_2 + l - 5 \le j_3}} 2^{(1-2b)j_1}2^{(1-2b)j_2}2^{-2bj_3} \lesssim 2^{(2-6b)(4k_2+l)},\]
 which implies
	\[\eqref{eq:hhl} \lesssim \sum_{k_2 \ge 10}\sum_{l\le 0}2^{(5-4s-24b)k_2}2^{(3-6b)l} \lesssim 1.\]
	
	If $j_3 \neq j_{max}$ ($j_2 = j_{max}$), similarly as before, we have
	\[\eqref{eq:hhl} \lesssim \sum_{\substack{k_2 \ge 10 \\ |k_1-k_2| \le 5}}\sum_{l\le 0}\sum_{0 \le j_1, j_3 \le j_2} 2^{-4sk_2}2^{2l}2^{-2b(j_1+j_2+j_3)}2^{(j_1+j_2+j_3)}2^{-3k_{max}}2^{-(k_2+j_2)}.\]
	Given $s > -2$, we can choose $\max((1-s)/6 , 1/3) < b < 1/2$.  Since $j_{max} \ge 4k_2 + l - 5$ and $1/3 <b < 1/2$, we have
	\[\sum_{\substack{0 \le j_1,j_3 \le j_2 \\4k_2 + l - 5 \le j_2}} 2^{(1-2b)j_1}2^{-2bj_2}2^{(1-2b)j_3} \lesssim 2^{(2-6b)(4k_2+l)},\]
	which implies
	\[\eqref{eq:hhl} \lesssim \sum_{k_2 \ge 10}\sum_{l\le 0}2^{(4-4s-24b)k_2}2^{(4-6b)l} \lesssim 1.\]
		
	\textbf{Case III-b} $k_3 \ge 1$. In this case, we have from \eqref{eq:support property} and \eqref{eq:resonant function} that $j_{max} \ge 4k_2 + k_3 - 5$. Similarly as \textbf{Case I}, it suffices to show
	\begin{equation}\label{eq:bilinear1-8}
	\sum_{\substack{k_2 \ge 10 \\ |k_1-k_2| \le 5 \\ 1 \le k_3 \le k_2 -5}}\sum_{j_1,j_2,j_3 \ge 0} 2^{2(1+s)k_3}2^{-4sk_2}2^{-2b(j_1+j_2+j_3)}2^{(j_1+j_2+j_3)}2^{-3k_{max}}2^{-(k_i+j_i)} \lesssim 1,
	\end{equation}
	thanks to \eqref{eq:weight} and Lemma \ref{lem:block estimate} (b). Without loss of generality, we may assume that $j_1 \le j_2$. 
	
	If $j_3 \neq j_{max}$ ($j_2 = j_{max}$),  given $ -2 < s$, we can choose $\max((1-s)/6 , (4-s)/15) < b < 1/2$.  Performing the summation over $0 \le j_1, j_3 \le j_2$ in addition to $4k_2 + k_3 - 5 \le j_2$ after choosing $(k_i,j_i) = (k_2,j_2)$ yields
	\[\mbox{LHS of } \eqref{eq:bilinear1-8} \lesssim  \sum_{k_2 \ge 10 }\sum_{1 \le k_3 \le k_2 -5} 2^{(4-4s-24b)k_2}2^{(4+2s-6b)k_3}\lesssim 1.\]
	
	If $j_3 = j_{max}$,  given $-7/4 < s$, we can choose $ \max((5-4s)/24, (4-s)/15) < b < 1/2$.  Performing the summation over $0 \le j_1 \le j_2 \le j_3$ in addition to $4k_2 + k_3 - 5 \le j_3$ after choosing $(k_i,j_i) = (k_3,j_3)$ yields
	\[\mbox{LHS of } \eqref{eq:bilinear1-8} \lesssim  \sum_{\substack{k_2 \ge 10 \\ |k_1-k_2| \le 5 \\ 1 \le k_3 \le k_2 -5}} 2^{(5-4s-24b)k_2}2^{(3+2s-6b)k_3} \lesssim 1,\] 
	which completes the proof of \eqref{eq:bilinear1-8}.
	
The \emph{low $\times$low $\Rightarrow$low} interaction component can be directly controlled by the Cauchy-Schwarz inequality, since  the low frequency localized space $D^{\alpha}$ with $\alpha >1/2$ allows the $L^2$ integrability with respect to $\tau$-variables.
	
	Therefore, the proof of \eqref{eq:bilinear1} is completed.
\end{proof}

\subsection{Control of $Y^{s,-b}$ norm of nonlinear terms}
\begin{proposition}\label{prop:bi2}
	For $-7/4 < s < 5/2$, there exists $b = b(s) < 1/2$ such that for all $\alpha > 1/2$, we have
	\begin{equation}\label{eq:bilinear2}
	\norm{\px(uv)}_{Y^{s,-b}} \le c\norm{u}_{X^{s,b} \cap D^{\alpha}}\norm{v}_{X^{s,b} \cap D^{\alpha}}.
	\end{equation}
\end{proposition}

The proof is similar as the proof of Lemma 5.1 (b) in \cite{Holmerkdv}, while we give a direct proof without the interpolation argument. We state the elementary integral estimates without proof.
\begin{lemma}[Lemmas 5.12, 5.13 \cite{Holmerkdv}]
Let $\alpha, \beta \in \R$. 

(a) If $\frac14 < b < \frac12$, then
\begin{equation}\label{eq:integral1}
\int_{-\infty}^{\infty} \frac{dx}{\bra{x-\alpha}^{2b}\bra{x-\beta}^{2b}} \le \frac{c}{\bra{\alpha-\beta}^{4b-1}}.
\end{equation}

(b) If $b < \frac12$, then
\begin{equation}\label{eq:integral2}
\int_{|x|\le\beta} \frac{dx}{\bra{x}^{4b-1}|\alpha-x|^{1/2}} \le \frac{c(1+\beta)^{2-4b}}{\bra{\alpha}^{1/2}}.
\end{equation}
\end{lemma}

\begin{proof}[Proof of Proposition \ref{prop:bi2}]
	We first address the range $-7/4 < s\le0$. If $|\tau| > \frac{1}{32}|\xi|^5$, it follows \eqref{eq:bilinear1} in the proof of Proposition \ref{prop:bi1} due to $\bra{\tau}^{\frac{s}{5}} \lesssim \bra{\xi}^s$. Thus, we may assume that $|\tau| \le \frac{1}{32}|\xi|^5$. We note that 
	\[\frac{31}{32}|\xi|^5 \le \frac{31}{32}|\xi|^5 - (|\tau| - \frac{1}{32}|\xi|^5) = |\xi|^5 - |\tau| \le |\tau-\xi^5| \le |\tau|+ |\xi|^5 \le \frac{33}{32}|\xi|^5,\]
	so that we have 
	\begin{equation}\label{eq:same size}
	|\tau-\xi^5| \sim |\xi|^5
	\end{equation} 
	under this assumption. Similarly, we also have
	\begin{equation}\label{eq:same size1}
	|\tau-\frac{1}{16}\xi^5| \sim |\xi|^5.
	\end{equation} 
	Then, by defining $f_i$ as in \eqref{eq:restriction1} under \eqref{eq:restriction2} and \eqref{eq:weight}, \eqref{eq:bilinear2} is equivalent to
	\begin{equation}\label{eq:bilinear2-1}
	\iint\limits_{\substack{\xi_1+\xi_2=\xi \\ \tau_1+\tau_2=\tau}} \frac{|\xi|\bra{\tau}^{\frac{s}{5}}\wt{f}_1(\tau_1,\xi_1)\wt{f}_2(\tau_2,\xi_2)\wt{f}_3(\tau,\xi)}{\bra{\xi_1}^s\bra{\xi_2}^s\bra{\xi}^{5b}\beta_1(\tau_1,\xi_1)\beta_2(\tau_2,\xi_2)} \lesssim \norm{f_1}_{L^2}\norm{f_2}_{L^2}\norm{f_3}_{L^2}.
	\end{equation}
	We may assume that $|\xi_1| \le |\xi_2|$ thanks to the symmetry.
	
	\textbf{Case I} $|\xi_2|< 1$. In this case, we know $|\xi| < 1$ due to the identity \eqref{eq:symmetry1}. Then, the left-hand side of \eqref{eq:bilinear2-1} is equivalent to
	\[\iint\limits_{\substack{\xi_1+\xi_2=\xi \\ \tau_1+\tau_2=\tau \\ |\xi_1|,|\xi_2|,|\xi| < 1}} \frac{\wt{f}_1(\tau_1,\xi_1)\wt{f}_2(\tau_2,\xi_2)\wt{f}_3(\tau,\xi)}{\bra{\tau_1}^{\alpha}\bra{\tau_2}^{\alpha}}\]
	due to $|\tau| \lesssim 1$. Since $\alpha > 1/2$, we have \eqref{eq:bilinear2-1} by the Cauchy-Schwarz inequality.
	
	\textbf{Case II} $|\xi_2| \ge 1$. 
	
	\textbf{Case II-1} $|\xi_1| < 1$. In this case, we can know that $|\xi| \ge 1$ due to the identity \eqref{eq:symmetry1}. Then, the left-hand side of \eqref{eq:bilinear2-1} is bounded by
	\begin{equation}\label{eq:bilinear2-2}
	\iint\limits_{\ast} \frac{|\xi|\wt{f}_1(\tau_1,\xi_1)\wt{f}_2(\tau_2,\xi_2)\wt{f}_3(\tau,\xi)}{\bra{\xi_2}^s\bra{\xi}^{5b}\bra{\tau_1}^{\alpha}}
	\end{equation}
	due to $\bra{\tau}^{s/5}\bra{\tau_2-\xi_2^5}^{-b} \lesssim 1$, where
	\[\ast = \{(\tau_1,\tau_2,\tau,\xi_1,\xi_2,\xi) \in \R^6: \xi_1+\xi_2=\xi,\;  \tau_1+\tau_2=\tau,\;  |\xi_1| < 1,\; |\xi_2|,|\xi| \ge 1\}.\]
	For given $-3/2 < s \le 0$, we can choose $b = b(s)$ satisfying $\frac{1-s}{5} < b < \frac12$. Since $|\xi_2| \sim |\xi|$, we have $|\xi|^{1-s-5b} \le 1$, and hence we have from the Cauchy-Schwarz inequality that
	\[\eqref{eq:bilinear2-2} \lesssim \norm{f_2}_{L^2}\norm{f_3}_{L^2} \iint\limits_{\set{|\xi_1| < 1} \times \R} \frac{\wt{f}_1(\tau_1,\xi_1)}{\bra{\tau_1}^{\alpha}} \; d\xi_1 \; d\tau_1\lesssim \norm{f_1}_{L^2}\norm{f_2}_{L^2}\norm{f_3}_{L^2}.\]
	The last inequality holds true due to $\alpha > 1/2$.
	
	\textbf{Case II-2} $1 \le |\xi_1| \le |\xi_2|$. We may further assume that $|\tau_1 - \xi_1^2| \le |\tau_2 - \xi_2^5|$ due to the symmetry. 
	
	\textbf{Case II-2.a} $|\tau_2 - \xi_2^5| \le 100000 |\tau -\xi^5|$. In this case, it suffices to show from the Cauchy-Schwarz inequality that
	\begin{equation}\label{eq:bilinear2-3}
	\sup_{\substack{\xi,\tau \in \R\\ |\tau| \le \frac{1}{32}|\xi|^5}}\Big(\iint\limits_{\substack{\xi_1+\xi_2=\xi \\ \tau_1+\tau_2=\tau}} \frac{|\xi|^2\bra{\tau}^{\frac{2s}{5}}}{\bra{\xi_1}^{2s}\bra{\xi_2}^{2s}\bra{\xi}^{10b}\bra{\tau_1-\xi_1^5}^{2b}\bra{\tau_2-\xi_2^5}^{2b}} \; d\xi_1\;d\tau_1\Big)^{1/2} \le c.
	\end{equation}
	Under the assumption, we only consider the case when $|\xi| \ge 1$. Otherwise, from \eqref{eq:same size}, $|\tau_1-\xi_1^5| \lesssim 1$, and hence we have \eqref{eq:bilinear2-1} similarly as \textbf{Case II} for $-2 \le s \le 0$. Indeed, from the identity \eqref{eq:symmetry2} under this condition, we know $|H| \lesssim 1$. Since 
	\[|H| = \frac52|\xi_1||\xi_2||\xi|(\xi_1^2+\xi_2^2+\xi^2) \ge 5|\xi_1|^2|\xi_2|^2|\xi|,\]
	we have $|\xi_1|^{-s}|\xi_2|^{-s} \lesssim |\xi|^{\frac{s}{2}}$, and hence $|\xi|^{1+s/2} \le 1$ for $-2 \le s \le 0$. The Cauchy-Schwarz inequality with respect to $\xi_1, \xi, \tau_1, \tau$ guarantees \eqref{eq:bilinear2-1}.
	
	We now consider \eqref{eq:bilinear2-3} on the case when $|\xi| \ge 1$. We use \eqref{eq:integral1} in addition to \eqref{eq:symmetry2} so that the left-hand side of \eqref{eq:bilinear2-3} is bounded by
	\begin{equation}\label{eq:bilinear2-4}
	\frac{|\xi|\bra{\tau}^{\frac{s}{5}}}{\bra{\xi}^{5b}}\left(\int_{\R} \frac{d\xi_1}{\bra{\xi_1}^{2s}\bra{\xi_2}^{2s}\bra{\tau-\xi^5 + H}^{4b-1}}\right)^{1/2}.
	\end{equation}
	The support property ($|\tau-\xi^5| \gtrsim |H|$) with \eqref{eq:same size} enables us to get
	\[|\xi_1|^{-2s}|\xi_2|^{-2s} \lesssim |\xi|^{-4s},\]
	and hence \eqref{eq:bilinear2-4} can be controlled by
	\begin{equation}\label{eq:bilinear2-5}
	\frac{|\xi|^{1-2s}\bra{\tau}^{\frac{s}{5}}}{\bra{\xi}^{5b}}\left(\int_{\R} \frac{d\xi_1}{\bra{\tau-\xi^5 + H}^{4b-1}}\right)^{1/2}.
	\end{equation}
	Let 
	\[\mu = \tau - \xi^5 + H.\]
	Note that $|\mu| \le 2|\tau-\xi^5|$ in this case. Then, by the direct calculation, we know
	\[\mu - (\tau - \frac{1}{16}\xi^5) = -\frac{5}{16}\xi(\xi-2\xi_1)^2(2\xi^2 + (\xi-2\xi_1)^2)\]
	and
	\[d\mu = \frac52\xi(\xi^2 + (\xi-2\xi)^2)(\xi-2\xi_1) \; d\xi_1.\]
	Since
	\[|\xi|^{\frac32}|\mu-(\tau-\frac{1}{16}\xi^5)|^{\frac12} \le |\xi||\xi-2\xi_1||2\xi^2 +(\xi-2\xi_1)^2| \le 2|\xi||\xi-2\xi_1||\xi^2 +(\xi-2\xi_1)^2|,\]
	we can reduce \eqref{eq:bilinear2-5} by
	\begin{equation}\label{eq:bilinear2-6}
	\frac{|\xi|^{1-2s}\bra{\tau}^{\frac{s}{5}}}{\bra{\xi}^{5b}|\xi|^{\frac34}}\left(\int_{|\mu| \lesssim |\tau - \xi^5|} \frac{d\mu}{\bra{\mu}^{4b-1}|\mu - (\tau-\frac{1}{16}\xi^5)|^{1/2}}\right)^{1/2}.
	\end{equation}
	By \eqref{eq:integral2}, \eqref{eq:bilinear2-6} is bounded by
	\[\frac{|\xi|^{1-2s}\bra{\tau}^{\frac{s}{5}}\bra{\tau-\xi^5}^{1-2b}}{\bra{\xi}^{5b}|\xi|^{\frac34}\bra{\tau-\frac{1}{16}\xi^5}^{1/4}}.\]
	For given $-7/4 < s \le 0$, we choose $b = b(s)$ satisfying $\frac{4-2s}{15} \le b < \frac12$. From \eqref{eq:same size} and \eqref{eq:same size1} with $|\xi| \ge 1$ and $s \le 0$, we obtain
	\[\frac{|\xi|^{1-2s}\bra{\tau}^{\frac{s}{5}}\bra{\tau-\xi^5}^{1-2b}}{\bra{\xi}^{5b}|\xi|^{\frac34}\bra{\tau-\frac{1}{16}\xi^5}^{1/4}} \lesssim |\xi|^{4-2s-15b} \lesssim 1.\]
	
	\textbf{Case II-2.b} $|\tau - \xi^5| \le \frac{1}{100000} |\tau_2 -\xi_2^5|$. In this case, it suffices to show from the Cauchy-Schwarz inequality that
	\begin{equation}\label{eq:bilinear2-7}
	\sup_{\xi_2,\tau_2 \in \R}\Big(\iint\limits_{\substack{\xi_1+\xi_2=\xi \\ \tau_1+\tau_2=\tau}} \frac{|\xi|^2\bra{\tau}^{\frac{2s}{5}}}{\bra{\xi_1}^{2s}\bra{\xi_2}^{2s}\bra{\xi}^{10b}\bra{\tau_1-\xi_1^5}^{2b}\bra{\tau_2-\xi_2^5}^{2b}} \; d\xi\;d\tau\Big)^{1/2} \le c.
	\end{equation}
	In this case we fix $-2 < s \le 0$. Since $-5/2 <-2 < s$, we can choose $b = b(s)$ satisfying $-s/5 \le b < \frac12$. From the fact that
	\[\bra{\tau}^{\frac{2s}{5}+2b} \lesssim \bra{\xi^5}^{\frac{2s}{5}+2b} \sim \bra{\xi}^{2s+10b},\]
	the left-hand side of \eqref{eq:bilinear2-7} is bounded by
	\begin{equation}\label{eq:bilinear2-8}
	\frac{1}{\bra{\xi_2}^{s}\bra{\tau_2-\xi_2^5}^{b}}\Big(\iint\limits_{\substack{\xi_1+\xi_2=\xi \\ \tau_1+\tau_2=\tau}} \frac{|\xi|^2\bra{\xi}^{2s}}{\bra{\xi_1}^{2s}\bra{\tau}^{2b}\bra{\tau_1-\xi_1^5}^{2b}} \; d\xi\;d\tau\Big)^{1/2}.
	\end{equation}
	
	When $|H| \le \frac12|\tau_2-\xi_2^5|$, we can know the following facts:
	\[\begin{aligned}
	&\hspace{-6em}\diamond |\tau-\xi^5| \ll |\tau_2-\xi_2^5| \; \mbox{ and }\; |\tau| \le \frac{1}{32}|\xi|^5 \; \mbox{ imply } \; |\xi|^5 \ll |\tau_2 - \xi_2^5|.\\
	&\hspace{-6em}\diamond \bra{\tau_2 - \xi_2^5 - H + \xi^5} \sim \bra{\tau_2 - \xi_2^5}.\\
	&\hspace{-6em}\diamond \bra{\xi_1}^{-2s}\bra{\xi_2}^{-2s} \lesssim |\tau_2 - \xi_2^5|^{-s}|\xi|^s.
	\end{aligned}\]
We perform the integration in \eqref{eq:bilinear2-8} in terms of $\tau$ variable by using \eqref{eq:integral1}, then \eqref{eq:bilinear2-8} is bounded by
	\[\begin{aligned}
	&\frac{1}{\bra{\xi_2}^{s}\bra{\tau_2-\xi_2^5}^{b}}\left(\int_{\R} \frac{|\xi|^2\bra{\xi}^{2s}}{\bra{\xi_1}^{2s}\bra{\tau_2 - \xi_2^5 - H + \xi^5}^{4b-1}} \; d\xi\right)^{1/2}\\
	&\sim \frac{1}{\bra{\xi_2}^{s}\bra{\tau_2-\xi_2^5}^{3b-1/2}}\left(\int_{|\xi| \le |\tau_2 - \xi_2^5|^{1/5}} \frac{|\xi|^2\bra{\xi}^{2s}}{\bra{\xi_1}^{2s}} \; d\xi\right)^{1/2}\\
	&\lesssim \frac{|\tau_2-\xi_5^5|^{-s/2}}{\bra{\tau_2-\xi_2^5}^{3b-1/2}}\left(\int_{|\xi| \le |\tau_2 - \xi_2^5|^{1/5}} |\xi|^{2+s}\bra{\xi}^{2s} \; d\xi\right)^{1/2}.
	\end{aligned}\]
	For given $-7/2 <-2< s \le 0$, we choose can $b = b(s)$ satisfying $\frac{4-s}{15} < b < \frac12$\footnote{The strict inequality $\frac{4-s}{15} < b$ covers the logarithmic divergence when $s = -1$.}. Then, by performing integration in terms of $\xi$, we have
	\[\frac{|\tau_2-\xi_5^5|^{-s/2}}{\bra{\tau_2-\xi_2^5}^{3b-1/2}}\left(\int_{|\xi| \le |\tau_2 - \xi_2^5|^{1/5}} |\xi|^{2+s}\bra{\xi}^{2s} \; d\xi\right)^{1/2} \lesssim \bra{\tau_2-\xi_2^5}^{\frac{1}{10}(8-30b-2s)}\lesssim 1.\]
	
	For the other case ($|H| > \frac12|\tau_2-\xi_2^5|$), we can know the following facts:
	\begin{equation}\label{eq:facts}
	\begin{aligned}
	&\hspace{-17em}\diamond 10|\xi| \le |\xi_1| \sim |\xi_2|.\\
	&\hspace{-17em}\diamond |\xi - \xi_2| \sim |\xi_2|.\\
	&\hspace{-17em}\diamond |\xi| \sim \frac{|\tau_2 - \xi_2^5|}{|\xi_2|^4}.\\
	&\hspace{-17em}\diamond |\xi|^5 \ll |\tau_2 - \xi_2^5|.\\
	&\hspace{-17em}\diamond \bra{\xi_1}^{-2s}\bra{\xi_2}^{-2s} \lesssim |\tau_2 - \xi_2^5|^{-s}|\xi|^s.
	\end{aligned}
	\end{equation}
	To verify the first one in \eqref{eq:facts}\footnote{It is not difficult to verify the others.}, suppose that $|\xi_1|\le 10|\xi|$. From \eqref{eq:symmetry1}, we know $|\xi_2| \le 11|\xi|$. Then,
	\[\begin{aligned}
	|H| &= \frac{5}{2}|\xi_1||\xi_2||\xi|(|\xi_1|^2 + |\xi_2|^2 + |\xi|^2) \\
	&\le 30525|\xi|^5 \le \frac{976800}{31}|\tau-\xi^5| \le \frac13|\tau_2 - \xi_2^5|,
	\end{aligned}\]
	which contradicts to the assumption $|H| > \frac12|\tau_2-\xi_2^5|$.
	
	Now, under the conditions \eqref{eq:facts}, we control the following integral:
	\begin{equation}\label{eq:bilinear2-10}
	\iint\limits_{\substack{\xi_1+\xi_2=\xi \\ \tau_1+\tau_2=\tau}} \frac{|\xi|^2\bra{\xi}^{2s}\bra{\xi_1}^{-2s}\bra{\xi_2}^{-2s}}{\bra{\tau_2-\xi_2^5}^{2b}\bra{\tau}^{2b}\bra{\tau_1-\xi_1^5}^{2b}} \; d\xi\;d\tau.
	\end{equation}
	
	When $|\xi| \le 1$, \eqref{eq:facts} and \eqref{eq:integral1} yield
	\[\begin{aligned}
	\eqref{eq:bilinear2-10} &\lesssim \iint\limits_{\substack{\xi_1+\xi_2=\xi \\ \tau_1+\tau_2=\tau}} \frac{|\xi|^{2+s}\bra{\xi}^{2s}\bra{\tau_2-\xi_2^5}^{-s-2b}}{\bra{\tau}^{2b}\bra{\tau_1-\xi_1^5}^{2b}} \; d\tau\;d\xi\\
	&\lesssim \int_{|\xi |\le 1} \frac{|\xi|^{2+s}\bra{\xi}^{2s}\bra{\tau_2-\xi_2^5}^{-s-2b}}{\bra{\tau_2-\xi_2^5 - H + \xi^5}^{4b-1}} \; d\xi
	\end{aligned}\]
	Let 
	\[\mu = \tau_2 - \xi_2^5 - H + \xi^5.\]
	A direct calculation gives
	\[d\mu = 5(\xi-\xi_2)^4 \; d\xi.\]
	From the facts \eqref{eq:facts} with $|\xi| \le 1$, since $|\xi_2|^{-4} \lesssim |\tau_2 - \xi_2^5|^{-1}$, the change of variable enables us to get
	\[\eqref{eq:bilinear2-10} \lesssim \int_{|\mu| \le |\tau_2 - \xi_2^5|} \frac{\bra{\tau_2-\xi_2^5}^{-s-2b-1}}{\bra{\mu}^{4b-1}} \; d\mu\]
	for $-2 < s \le 0$. For given $-2 < s \le 0$, we can choose $b=b(s)$ satisfying $\frac{1-s}{6} \le b < \frac12$. Then, by performing the integration in terms of $\mu$, we have
	\[\eqref{eq:bilinear2-10} \lesssim \bra{\tau_2-\xi_2^5}^{-s-6b+1} \lesssim 1.\]
	
	Now, we focus on the case when $|\xi| > 1$. Similarly as before, \eqref{eq:bilinear2-10} can be reduced by 
	\begin{equation}\label{eq:bilinear2-11}
	\int_{|\xi| > 1} \frac{|\xi|^{2+3s}\bra{\tau_2-\xi_2^5}^{-s-2b}}{\bra{\tau_2-\xi_2^5 - H + \xi^5}^{4b-1}} \; d\xi.
	\end{equation}
		
	We use the change of variable $\mu = \tau_2 - \xi_2^5 - H + \xi^5$ with
	\[d\mu = 5(\xi-\xi_2)^4 \; d\xi.\]
	If $-2 < s \le -1$, since 
	\[|\xi_2|^{-4} \sim |\xi||\tau_2-\xi_2^5|^{-1} \hspace{1em}(\Rightarrow |\xi|^{3+3s} \lesssim 1),\]
	we have 
	\[\eqref{eq:bilinear2-11} \lesssim \int_{|\mu| < |\tau_2-\xi_2^5|} \frac{\bra{\tau_2-\xi_2^5}^{-s-2b-1}}{\bra{\mu}^{4b-1}} \; d\mu \lesssim\bra{\tau_2-\xi_2^5}^{-s-6b+1} \lesssim 1\]
	by choosing $b = b(s)$ satisfying $(1-s)/6 < b < 1/2$.
	
	Otherwise ($-1 < s \le 0$), we can choose $b=b(s)$ satisfying $\frac{4-s}{15} \le b < \frac12$. Then, from the fact $|\xi| \ll |\tau_2 - \xi_2^5|^{1/5}$, we obtain 
	\[\eqref{eq:bilinear2-11} \lesssim \int_{|\mu| \le |\tau_2 - \xi_2^5|} \frac{\bra{\tau_2-\xi_2^5}^{\frac{3+3s}{5}-s-2b-1}}{\bra{\mu}^{4b-1}} \; d\mu \lesssim \bra{\tau_2 - \xi_2^5}^{\frac{8-2s-30b}{5}} \lesssim 1.\]
	
	Therefore, we complete the proof of \eqref{eq:bilinear2} for $-7/4 < s \le 0$.
	
	Now we address the range $0 < s < 5/2$. For the positive regularity, it is enough to consider $|\tau| \ge 2|\xi|^5$, otherwise, it follows \eqref{eq:bilinear1} in the proof of Proposition \ref{prop:bi1}. Note that $|\tau| \ge 2|\xi|^5$ enables us to assume  that $\bra{\tau-\xi^5} \sim \bra{\tau}$. We further assume that $|\xi_1| \le |\xi_2|$ by the symmetry. Then, similarly as before (for the negative regularity range) it suffices to show
	\begin{equation}\label{eq:bilinear2.1-1}
	\iint\limits_{\substack{\xi_1+\xi_2=\xi \\ \tau_1+\tau_2=\tau}} \frac{|\xi|\bra{\tau}^{\frac{s}{5}}\wt{f}_1(\tau_1,\xi_1)\wt{f}_2(\tau_2,\xi_2)\wt{f}_3(\tau,\xi)}{\bra{\xi_1}^s\bra{\xi_2}^s\bra{\tau}^{b}\beta_1(\tau_1,\xi_1)\beta_2(\tau_2,\xi_2)} \lesssim \norm{f_1}_{L^2}\norm{f_2}_{L^2}\norm{f_3}_{L^2},
	\end{equation}
	where $f_i$ and $\beta_i$, $i=1,2$ are defined as in \eqref{eq:restriction1} and \eqref{eq:restriction2}, respectively.
	
	\textbf{Case I} $|\xi_2| \le 1$. In this case, the left-hand side of \eqref{eq:bilinear2.1-1} can be reduced by
	\begin{equation}\label{eq:bilinear2.1-2}
	\iint\limits_{\substack{\xi_1+\xi_2=\xi \\ \tau_1+\tau_2=\tau\\|\xi_1|,|\xi_2|,|\xi| \le 1}} \frac{\bra{\tau}^{\frac{s}{5}-b}\wt{f}_1(\tau_1,\xi_1)\wt{f}_2(\tau_2,\xi_2)\wt{f}_3(\tau,\xi)}{\bra{\tau_1}^{\alpha}\bra{\tau_2}^{\alpha}}.
	\end{equation}
	For $0 < s < 5/2$, there exists $b=b(s)$ satisfying $s/5 < b < 1/2$ such that
	\[\eqref{eq:bilinear2.1-2} \lesssim \iint\limits_{\substack{\xi_1+\xi_2=\xi \\ \tau_1+\tau_2=\tau\\|\xi_1|,|\xi_2|,|\xi| \le 1}} \frac{\wt{f}_1(\tau_1,\xi_1)\wt{f}_2(\tau_2,\xi_2)\wt{f}_3(\tau,\xi)}{\bra{\tau_1}^{\alpha}\bra{\tau_2}^{\alpha}} \lesssim \norm{f_1}_{L^2}\norm{f_2}_{L^2}\norm{f_3}_{L^2},\]
	by the Cauchy-Schwarz inequality, due to $\alpha > 1/2$.
	
	\textbf{Case II} $|\xi_2| > 1$.
	
	\textbf{Case II-a} $|\xi_1| \le 1$. In this case, we know $|\xi|>1$, due to the identity \eqref{eq:symmetry1}. Then, the left-hand side of \eqref{eq:bilinear2.1-1} is reduced to
	\begin{equation}\label{eq:bilinear2.1-3}
	\iint\limits_{\ast} \frac{\bra{\xi}^{1-s}\bra{\tau}^{\frac{s}{5}-b}\wt{f}_1(\tau_1,\xi_1)\wt{f}_2(\tau_2,\xi_2)\wt{f}_3(\tau,\xi)}{\bra{\tau_1}^{\alpha}},
	\end{equation}
	where
	\[\ast = \{(\tau_1,\tau_2,\tau,\xi_1,\xi_2,\xi) \in \R^6: \xi_1+\xi_2=\xi \; \tau_1+\tau_2=\tau\;|\xi_1| \le 1 \;|\xi_2|\sim|\xi| > 1\}.\]
	If $1-s >0$, we can choose $1/5 < b <1/2$ such that
	\begin{equation}\label{eq:bilinear2.1-4}
	\eqref{eq:bilinear2.1-3} \lesssim  \norm{f_1}_{L^2}\norm{f_2}_{L^2}\norm{f_3}_{L^2}
	\end{equation}
	holds true by the Cauchy-Schwarz inequality, due to $\alpha > 1/2$. Otherwise ($1 \le s <5/2$), there exists $b=b(s)$ satisfying $s/5 < b < 1/2$ such that \eqref{eq:bilinear2.1-4} holds true.
	
	\textbf{Case II-b} $|\xi_1| > 1$. By the Cauchy-Schwarz inequality, it suffices to show
	\begin{equation}\label{eq:bilinear2.1-5}
	\sup_{\xi,\tau \in \R}\Big(\iint\limits_{\substack{\xi_1+\xi_2=\xi \\ \tau_1+\tau_2=\tau\\1<|\xi_1|\le |\xi_2|}} \frac{|\xi|^2\bra{\tau}^{\frac{2s}{5}-2b}}{\bra{\xi_1}^{2s}\bra{\xi_2}^{2s}\bra{\tau_1-\xi_1^5}^{2b}\bra{\tau_2-\xi_2^5}^{2b}} \; d\xi_1\;d\tau_1\Big)^{1/2} \le c.
	\end{equation}
	By \eqref{eq:integral1}, the left-hand side of \eqref{eq:bilinear2.1-5} is bounded by
	\begin{equation}\label{eq:bilinear2.1-6}
	|\xi|\bra{\tau}^{\frac{s}{5}-b}\Big(\int\limits_{|\xi_1|>1} \frac{1}{\bra{\xi_1}^{2s}\bra{\xi_2}^{2s}\bra{\tau - \xi^5 + H}^{4b-1}} \; d\xi_1\Big)^{1/2}.
	\end{equation}
	
	If $1 < |\xi_1| \ll |\xi_2|$, we know that
	\[|\xi_2| \sim |\xi|, \hspace{1em} |H| \sim |\xi_1|\xi^4 \ll |\xi|^5 \le \frac12|\tau| \; \Rightarrow \; |\xi_1| \ll |\tau||\xi|^{-4}.\]
	For $0 < s < 5/2$, there exists $b = b(s)$ satisfying
	\[\frac13\left(\frac{s}{5} + 1\right) < b < \frac12\]
	such that
	\[\eqref{eq:bilinear2.1-6} \lesssim |\xi|^{-1-s}\bra{\tau}^{\frac{s}{5}-3b+1}\Big(\int\limits_{|\xi_1|>1} \frac{1}{|\xi_1|^{2s+1}} \; d\xi_1\Big)^{1/2} \le c\]
	holds true.
	
	Let us consider the other case ($1<|\xi_1| \sim |\xi_2|$). Note from \eqref{eq:symmetry1} that $|\xi| \lesssim |\xi_1| \sim |\xi_2|$. We further divide this case into the following three cases:
	\[\mbox{(a) } |H| \ll |\tau|, \hspace{1em} \mbox{(b) } |\tau| \ll |H|, \hspace{1em} \mbox{(c) } |H| \sim |\tau|.\]
	
	In the case of (a), we know $\bra{\tau - \xi^5 + H} \sim \bra{\tau}$. Then \eqref{eq:bilinear2.1-6} is reduced to
	\[|\xi|\bra{\tau}^{\frac{s}{5}-3b + \frac12}\Big(\int\limits_{|\xi_1|>1} \frac{1}{\bra{\xi_1}^{4s}} \; d\xi_1\Big)^{1/2}.\]
	Since $\bra{\xi} \lesssim \bra{\xi_1}$ and $|\xi|^3|\xi_1|^2 \lesssim |H| \ll |\tau|$, we have
	\begin{equation}\label{eq:bilinear2.1-8}
	|\xi|^{\frac14}\bra{\xi}^{-s}\bra{\tau}^{\frac{s}{5}-3b + \frac34}\Big(\int\limits_{|\xi_1|>1} \frac{1}{\bra{\xi_1}^{1+2s}} \; d\xi_1\Big)^{1/2}.
	\end{equation}
	If $|\xi| \le 1$, for $0 < s < 15/4$, there exists $b = b(s)$ satisfying 
	\[\frac13\left( \frac{s}{5} + \frac34 \right) < b < \frac12\]
	such that \eqref{eq:bilinear2.1-8} is bounded. Otherwise, for $0 < s < 15/4$, there exists $b = b(s)$ satisfying 
	\[\max\left[\frac13\left( \frac{s}{5} + \frac34 \right), \frac{4}{15} \right] < b < \frac12\]
	such that \eqref{eq:bilinear2.1-8} is bounded.
	
	In the case of (b), we know $\bra{\tau - \xi^5 + H} \sim \bra{H}$. It is enough to consider $|H| \gg 1$, otherwise, we have further smoothing effects from the low frequency derivative. Indeed, due to $|\xi||\xi_1|^4 \lesssim |H| \lesssim 1$, \eqref{eq:bilinear2.1-6} is bounded by
	\begin{equation}\label{eq:bilinear2.1-9}
	\bra{\tau}^{\frac{s}{5}-b}\Big(\int\limits_{|\xi_1|>1} \frac{1}{\bra{\xi_1}^{4s+8}} \; d\xi_1\Big)^{1/2}.
	\end{equation}
	For $0 < s < 5/2$, there exists $b=b(s)$ satisfying $s/5 < b <1/2$ such that \eqref{eq:bilinear2.1-9} is bounded. Under the condition $|H| \gg 1$, note from $|\xi\xi_1^4| \lesssim |H|$ that
	\[\bra{\tau - \xi^5 + H}^{1-4b} \sim |H|^{1-4b} \lesssim |\xi_1|^{4-16b}|\xi|^{1-4b}\]
	for $1/4 < b < 1/2$. Then, \eqref{eq:bilinear2.1-6} is bounded by
	\begin{equation}\label{eq:bilinear2.1-10}
	\Big(\int\limits_{|\xi_1|>1} \frac{|\xi|^{3-4b}\bra{\tau}^{\frac{2s}{5}-2b}}{|\xi_1|^{4(s+4b-1)}} \; d\xi_1\Big)^{1/2}.
	\end{equation}
	Since $3-4b > 0$ implies $|\xi|^{3-4b} \lesssim |\xi_1|^{3-4b}$, for $0 < s < 5/2$, there exists $b=b(s)$ satisfying
	\[\max\left[\frac14,\frac{s}{5}, \frac{2-s}{5} \right] < b < \frac12\]
	such that 
	\[\eqref{eq:bilinear2.1-10} \lesssim \Big(\int\limits_{|\xi_1|>1} \frac{\bra{\tau}^{\frac{2s}{5}-2b}}{{|\xi_1|^{4(s+4b-1)+4b-3}}} \; d\xi_1\Big)^{1/2} \lesssim 1.\]
	
	In the case of (c), the term $\bra{\tau-\xi^5+H}^{4b-1}$ in the denominator of integrand in \eqref{eq:bilinear2.1-6} is always negligible when $b \ge 1/4$. We first consider the case when $1 \lesssim |\xi| \lesssim |\xi_1|$. Note that
	\[|\tau| \sim |H| \gtrsim |\xi||\xi_1|^4 \gtrsim1 \Rightarrow \bra{\tau} \sim |\tau|.\]
	From this, for $0 < s < 5/2$, there exists $b=b(s)$ satisfying 
	\[\max\left[\frac14, \frac{s}{5}, \frac{3-2s}{10}\right] < b <\frac12\]
	such that
	\[\begin{aligned}
	\eqref{eq:bilinear2.1-6} &\lesssim \Big(\int\limits_{|\xi_1|>1} \frac{|\xi|^2|\tau|^{\frac{2s}{5}-2b}}{|\xi_1|^{4s}} \; d\xi_1\Big)^{1/2} \lesssim \Big(\int\limits_{|\xi_1|>1} \frac{|\xi|^{2+\frac{2s}{5}-2b}|\xi_1|^{4(\frac{2s}{5}-2b)}}{|\xi_1|^{4s}} \; d\xi_1\Big)^{1/2}\\
	&\lesssim \Big(\int\limits_{|\xi_1|>1} \frac{1}{|\xi_1|^{2s+10b-2}} \; d\xi_1\Big)^{1/2} \lesssim 1,
	\end{aligned}\]
	since $2s/5 -2b +2 > 0$ and $2s+10b-2 > 1$.
	
	Now we finish the proof with considering the case when $|\xi| \le 1$. We further divide this case into two cases: $|\xi| \lesssim |\xi_1|^{-1/2}$ and $|\xi_1|^{-1/2} \ll |\xi| \le 1$. 
	
	In the case when $|\xi| \lesssim |\xi_1|^{-1/2}$, for $0 < s < 5/2$, there exists $b=b(s)$ satisfying $\max(1/4 , s/5) < b < 1/2$ such that
	\[\eqref{eq:bilinear2.1-6} \lesssim \Big(\int\limits_{|\xi_1|>1} \frac{|\xi_1|^{-1}}{|\xi_1|^{4s}} \; d\xi_1\Big)^{1/2} \lesssim 1.\]
	Otherwise ($|\xi_1|^{-1/2} \ll |\xi| \le 1$), note that
	\[|\xi_1|^{\frac72} \ll |\xi||\xi_1|^4 \lesssim |H| \sim |\tau|.\]
	Hence, for $0 < s < 5/2$, there exists $b=b(s)$ satisfying 
	\[\max\left[\frac14 , \frac{s}{5}\right] < b < 1/2\]
	such that
	\[\eqref{eq:bilinear2.1-6} \lesssim \Big(\int\limits_{|\xi_1|>1} \frac{|\xi_1|^{\frac72(\frac{2s}{5}-2b)}}{|\xi_1|^{4s}} \; d\xi_1\Big)^{1/2} \lesssim 1.\]
	The last inequality holds due to $\frac17(1-\frac{13s}{5}) < 1/4$ ($\leftrightarrow 13s/5 + 7b > 0$) for $s>0$.
\end{proof}

\section{Proof of Theorem \ref{theorem1}}\label{sec:main proof 1}

Here we shall prove Theorems \ref{theorem1}. We follow the arguments in \cite{Holmerkdv} (see also \cite{Cavalcante} and references therein). 

We fix $-\frac74<s<\frac52$ and  $s\neq \frac12,\ \frac32 $. Form \eqref{eq:scaling}, we may assume 
$$\|u_0\|_{H^s(\mathbb{R}^+)}+\|f\|_{H^{\frac{s+2}{5}}(\mathbb{R}^+)}+\|g\|_{H^{\frac{s+1}{5}}(\mathbb{R}^+)}=\delta,$$ 
for $\delta$ sufficiently small.

Select an extension $\underline{u}_0\in H^s(\mathbb{R})$ of $u_0$ such that $\|\underline{u}_0\|_{H^s(\mathbb{R})}\leq 2\|u_0\|_{H^s(\mathbb{R}^+)}$. Let $b=b(s)<\frac{1}{2}$ and $\alpha=\alpha(s)>\frac12$  such that the estimates given in Propositions \ref{prop:bi1} and \ref{prop:bi2} are valid. Together with arguments in Subsections \ref{section4} and \ref{section4-1}, let
\begin{equation}\label{eq:solution}
u(t,x)= \mathcal{L}_+^{\lambda_1}\gamma_1(t,x)+\mathcal{L}_+^{\lambda_2}\gamma_2(t,x)+F(t,x),
\end{equation}
where $\gamma_i$ ($i=1,2$) will be chosen in the following in terms of given initial and boundary data $u_0$, $f$ and $g$, and $F(t,x)=e^{it\partial_x^5}\underline{u}_0 +\mathcal{D}w(t,x)$.

Recall \eqref{eq:entries} in Subsection \ref{section4} by 
\[a_j = \frac{1}{5B(0)\Gamma\left(\frac45\right)}\frac{\cos\left(\frac{(1+4\lambda_j)\pi}{10}\right)}{\sin\left(\frac{(1-\lambda_j)\pi}{5}\right)}\quad \text{and}\quad b_j = \frac{1}{5B(0)\Gamma\left(\frac45\right)}\frac{\cos\left(\frac{(4\lambda_j-3)\pi}{10}\right)}{\sin\left(\frac{(2-\lambda_j)\pi}{5}\right)}.\]
Note that $a_j$ and $b_j$ are well-defined when even $\lambda_j = 1-5n$ or $\lambda_j = 2-5n$ for $n \in \Z$.

By Lemmas \ref{holmer1} and \ref{trace1}, we get
\begin{equation}\label{eq:f}
f(t)= u(t,0) = a_1\gamma_1(t)+a_2\gamma_2(t)+F(t,0)
\end{equation}
and
\begin{equation}\label{eq:g}
g(t)=\partial_x u(t,0) =b_1\mathcal{I}_{-\frac15}\gamma_1(t)+b_2\mathcal{I}_{-\frac15}\gamma_2(t)+ \partial_xF(t,0).
\end{equation}
Together with \eqref{eq:f} and \eqref{eq:g}, we can write as a matrix form 
\[ \left[\begin{array}{c}
f(t)-F(t,0) \\
\mathcal{I}_{\frac15}g(t)- \mathcal{I}_{\frac15}\partial_x F(t,0)  \end{array} \right]=
A
 \left[\begin{array}{c}
\gamma_1(t)\\
\gamma_2(t)   \end{array} \right], \]
where
\[A(\lambda_1,\lambda_2)=
\left[\begin{array}{cc}
a_1 & a_2 \\
	b_1 & b_2 \end{array} \right].\]\
	
For $-\frac74<s<\frac52$ and  $s\neq \frac12,\ \frac32 $, the choice of parameters $\lambda_1$ and $\lambda_2$ satisfying the following conditions
\begin{equation}\label{eq:lambda condition}
\max(s-2,\ -3)<\lambda_j<\min\left(\frac12,s+\frac12\right),\; j=1,2,
\end{equation}
holds Lemma \ref{edbf}.

On the other hand, a direct calculation shows that the matrix $A(\lambda_1, \lambda_2)$ is invertible, when
\begin{equation}\label{eq:lambda condition1}
\lambda_1 - \lambda_2 \neq 5n, \qquad n \in \Z.
\end{equation}
From this observation, we can define the forcing functions $\gamma_1(t)$ and $\gamma_2(t)$ for any $\lambda_j$, $j=1,2$, satisfying \eqref{eq:lambda condition1} by
\begin{equation}\label{lambda} \left[\begin{array}{c}
\gamma_1(t)\\
\gamma_2(t)   \end{array} \right]=A^{-1}\left[\begin{array}{c}
f(t)-F(t,0) \\
\mathcal{I}_{\frac15}g(t)-  \mathcal{I}_{\frac15} \partial_xF(t,0)  \end{array} \right],
\end{equation}
which shows \eqref{eq:solution} satisfies $(\pt - \px^5)u = w$.

We choose $\lambda_1(s), \lambda_2(s)$ satisfying \eqref{eq:lambda condition} and \eqref{eq:lambda condition1} so that $A(\lambda_1, \lambda_2)$ is invertible. Define the solution operator on $[0,1]$ by\footnote{We may use $\psi$ defined in \eqref{eq:cutoff} by restricting on $[0,\infty)$.}
\begin{equation}\label{eq:solution op}
\Lambda u(t,x)= \psi(t)\mathcal{L}_+^{\lambda_1}\gamma_1(t,x)+\psi(t)\mathcal{L}_+^{\lambda_2}\gamma_2(t,x)+\psi(t)F(t,x),
\end{equation}
where 
\[\left[\begin{array}{c}
\gamma_1(t)\\
\gamma_2(t)   \end{array} \right]=A^{-1}\left[\begin{array}{c}
f(t)-F(t,0) \\
\mathcal{I}_{\frac15}g(t)-  \mathcal{I}_{\frac15} \partial_xF(t,0)  \end{array} \right],\]
and $F(t,x)=e^{it\partial_x^5}\underline{u}_0-\mathcal{D}(\partial_x(u^2))(t,x)$. 
 
We remark in view of \eqref{eq:f}, \eqref{eq:g} and \eqref{lambda} that  it is necessary to check $\gamma_i(t)$, $i=1,2$ to be well-defined in $H_0^{\frac{s+2}{5}}(\R^+)$ thanks to Lemmas \ref{edbf}. It follows from Lemmas \ref{grupo}, \ref{duhamel} and \ref{edbf}, Propositions \ref{prop:bi1} and \ref{prop:bi2} and Lemmas \ref{sobolevh0} and \ref{alta}. We omit the details and refer to \cite{Holmerkdv}.
 
Recall the solution space $Z_1^{s,\alpha,b}$ defined in Subsection \ref{sec:sol space} under the norm
\[\norm{v}_{Z_1^{s,b,\alpha}} = \sup_{t \in \R} \norm{v(t,\cdot)}_{H^s} + \sum_{j=0}^{1}\sup_{x \in \R} \norm{\px^jv(\cdot,x)}_{H^{\frac{s+2-j}{5}}} + \norm{v}_{X^{s,b} \cap D^{\alpha}}.\]

All estimates obtained in Sections \ref{sec:energy} and \ref{sec:bilinear} yield
\[\|\Lambda u\|_{Z_1^{s,\alpha,b}}\leq c(\|u_0\|_{H^s(\R^+)}+\|f\|_{H^{\frac{s+2}{5}}(\R^+)}+\|g\|_{H^{\frac{s+1}{5}(\R^+)}} ) + C_1\norm{u}_{Z_1^{s,\alpha,b}}^2.\]
Similarly,
\[\|\Lambda u_1 - \Lambda u_2\|_{Z_1^{s,\alpha,b}}\leq C_2(\norm{u_1}_{Z_1^{s,\alpha,b}}+\norm{u_2}_{Z_1^{s,\alpha,b}})\norm{u_1-u_2}_{Z_1^{s,\alpha,b}},\]
for $u_1(0,x) = u_2(0,x)$.

Once we choose $0 < \delta \ll 1$ satisfying
\[4cC_1\delta < 1 \quad \mbox{and} \quad 4cC_2\delta < \frac12,\]
we can show $\Lambda$ is a contraction map on $\set{u \in Z_1^{s,\alpha,b} : \norm{u}_{Z_1^{s,\alpha,b}} < 2c\delta}$, and it completes the proof.

\section{Proof of Theorem \ref{theorem12}}\label{sec:main proof 2}
We follows the same argument used in Section \ref{sec:main proof 1}, and hence it suffices to construct solution operator similarly as \eqref{eq:solution op} and to find suitable functions $\gamma_i$ as in \eqref{lambda}.

We seek $\gamma_i$, $i=1,2,3$, in terms of given initial and boundary data, which satisfy the following solution form:
\begin{equation}\label{eq:solution2} 
u(t,x)= \mathcal{L}_-^{\lambda_1}\gamma_1(t,x)+\mathcal{L}_-^{\lambda_2}\gamma_2(t,x)+\mathcal{L}_-^{\lambda_3}\gamma_3(t,x)+F(t,x).
\end{equation} 
Let $a_j$, $b_j$ and $c_j$ be constants depending on $\lambda_j$, $j=1,2,3$, given by
\[a_j = \frac{\cos\left(\frac{(1-6\lambda_j)\pi}{10}\right)}{5B(0)\Gamma\left(\frac45\right)\sin\left(\frac{(1-\lambda_j)\pi}{5}\right)},\; b_j = \frac{\cos\left(\frac{(7-6\lambda_j)\pi}{10}\right)}{5B(0)\Gamma\left(\frac45\right)\sin\left(\frac{(2-\lambda_j)\pi}{5}\right)}\;
\text{and}\;
c_j = \frac{\cos\left(\frac{(13-6\lambda_j)\pi}{10}\right)}{5B(0)\Gamma\left(\frac45\right)\sin\left(\frac{(3-\lambda_j)\pi}{5}\right)}.\] 
Note that $a_j$, $b_j$ and $c_j$ are well-defined when even $\lambda_j = 1-5n$ or $\lambda_j = 2-5n$ or $\lambda_j = 3-5n$ for $n \in \Z$.

By Lemmas \ref{holmer1} and \ref{trace1}, we get
\begin{equation}\label{eq:f2}
f(t)= u(t,0) = a_1\gamma_1(t)+a_2\gamma_2(t)+a_3\gamma_3(t)+F(t,0),
\end{equation}
\begin{equation}\label{eq:g2}
g(t)=\partial_x u(t,0) =b_1\mathcal{I}_{-\frac15}\gamma_1(t)+b_2\mathcal{I}_{-\frac15}\gamma_2(t)+b_3\mathcal{I}_{-\frac15}\gamma_3(t)+ \partial_xF(t,0)
\end{equation}
and
\begin{equation}\label{eq:h2}
h(t)=\partial_x^2 u(t,0) =c_1\mathcal{I}_{-\frac25}\gamma_1(t)+c_2\mathcal{I}_{-\frac25}\gamma_2(t)+c_3\mathcal{I}_{-\frac25}\gamma_3(t)+\partial_x^2F(t,0).
\end{equation}

Together with \eqref{eq:f2}, \eqref{eq:g2} and \eqref{eq:h2}, we can express $(\gamma_1,\gamma_2,\gamma_2)$ as the following form 
\[ \left[\begin{array}{c}
f(t)-F(t,0) \\
\mathcal{I}_{\frac15}g(t)- \mathcal{I}_{\frac15} \partial_x F(t,0) \\
\mathcal{I}_{\frac25}h(t)-  \mathcal{I}_{\frac25} \partial_x^2 F(t,0)
 \end{array} \right]=
A
\left[\begin{array}{c}
\gamma_1(t)\\
\gamma_2(t)\\
\gamma_3(t)   \end{array} \right], \]

where
\[A(\lambda_1,\lambda_2,\lambda_3)=
\left[\begin{array}{ccc}
a_1 & a_2& a_3 \\
b_1 & b_2&b_3\\
c_1 & c_2&c_3
 \end{array} \right].\]\

For $-\frac74<s<\frac52$ and  $s\neq \frac12,\ \frac32 $, the choice of parameters $\lambda_1$ and $\lambda_2$ satisfying the following conditions
\[\max(s-2,\ -2)<\lambda_j<\min\left(\frac12,s+\frac12\right),\; j=1,2,3,\]
holds Lemma \ref{edbf}.

On the other hand, a direct calculation shows that the matrix $A(\lambda_1, \lambda_2, \lambda_3)$ is invertible, when
\begin{equation}\label{eq:lambda condition4}
\lambda_1 -\lambda_2 \neq 5n \quad \mbox{and} \quad \lambda_2 -\lambda_3 \neq 5n \quad \mbox{and} \quad \lambda_3 -\lambda_1 \neq 5n \quad n\in \Z.
\end{equation}
From this observation, we can define the forcing functions $\gamma_1(t)$ and $\gamma_2(t)$ for any $\lambda_j$, $j=1,2,3$, satisfying \eqref{eq:lambda condition4} by

\[\left[\begin{array}{c}
\gamma_1(t)\\
\gamma_2(t)\\
\gamma_3(t)   \end{array} \right]=A^{-1}\left[\begin{array}{c}
f(t)-F(t,0) \\
\mathcal{I}_{\frac15}g(t)-\partial_x  \mathcal{I}_{\frac15} F(t,0) \\
\mathcal{I}_{\frac25}h(t)-\partial_x^2  \mathcal{I}_{\frac25} \partial_xF(t,0)
\end{array} \right],\]
which shows \eqref{eq:solution2} satisfies $(\pt - \px^5)u + \px(u^2) = 0$, in the sense of distributions, and $u(t,0)=f(t)$, $\partial_xu(t,0)=g(t)$ and $\partial_x^2u(t,0)=h(t)$.

We then perform the same argument used in the proof of Theorem \ref{theorem1} to complete the proof of Theorem \ref{theorem12}.

\end{document}